\documentclass[10pt,reqno]{amsart}
\usepackage{amssymb,mathrsfs,graphicx,amsmath}
\usepackage{ifthen}
\usepackage{graphicx,extpfeil}
\usepackage{colortbl}
\usepackage{colortbl}
\definecolor{black}{rgb}{0.0, 0.0, 0.0}
\definecolor{red}{rgb}{1.0, 0.5, 0.5}
\provideboolean{shownotes} 
\setboolean{shownotes}{true} 
\newcommand{\margnote}[1]{
\ifthenelse{\boolean{shownotes}}%
{\marginpar{\raggedright\tiny\texttt{#1}}}%
{}%
}
\newcommand{\hole}[1]{
\ifthenelse{\boolean{shownotes}}%
{\begin{center} \fbox{ \rule {.25cm}{0cm} \rule[-.1cm]{0cm}{.4cm}
\parbox{.85\textwidth}{\begin{center} \texttt{#1}\end{center}} \rule
{.25cm}{0cm}}\end{center}} {} }

\title[Pressureless Euler/compressible Navier-Stokes equations]{The Cauchy problem for the pressureless Euler/isentropic Navier-Stokes equations}

\author[Choi]{Young-Pil Choi}
\address[Young-Pil Choi]{\newline Fakult\"at f\"ur Mathematik \newline
Technische Universit\"at M\"unchen, Boltzmannstr\ss e, 85748, Garching bei M\"unchen, Germany}
\email{ychoi@ma.tum.de}

\author[Kwon]{Bongsuk Kwon}
\address[Bongsuk Kwon]{\newline Department of Mathematical Sciences, School of Natural Science \newline
Ulsan National Institute of Science and Technology, Ulsan 689-798, Korea}
\email{bkwon@unist.ac.kr}

\numberwithin{equation}{section}

\newtheorem{theorem}{Theorem}[section]
\newtheorem{lemma}{Lemma}[section]
\newtheorem{corollary}{Corollary}[section]
\newtheorem{proposition}{Proposition}[section]
\newtheorem{remark}{Remark}[section]

\newcommand{\R}{\mathbb R}

\newcommand{\ls}{\lesssim}

\newcommand{\T}{\mathbb T}
\newcommand{\N}{\mathbb N}

\newcommand{\bq}{\begin{equation}}
\newcommand{\eq}{\end{equation}}

\newcommand{\lt}{\left}
\newcommand{\rt}{\right}
\newcommand{\na}{q}
\newcommand{\mc}{\mathcal{C}}

\newcommand{\pa}{\partial}

\newcommand{\mi}{\mathcal{I}}
\newcommand{\mn}{\mathcal{N}}
\newcommand{\lag}{\langle}
\newcommand{\rag}{\rangle}
\newcommand{\rmmm}{\rho^{m+1,m}}
\newcommand{\rmm}{\rho^{m,m-1}}
\newcommand{\ummm}{u^{m+1,m}}
\newcommand{\umm}{u^{m,m-1}}
\newcommand{\wmmm}{w^{m+1,m}}
\newcommand{\wmm}{w^{m,m-1}}
\newcommand{\vmmm}{v^{m+1,m}}
\newcommand{\vmm}{v^{m,m-1}}
\newcommand{\nmmm}{n^{m+1,m}}

\def\charf {\mbox{{\text 1}\kern-.30em {\text l}}}

\begin{document}
\allowdisplaybreaks

\date{\today}




\begin{abstract}
We present a new hydrodynamic model consisting of the pressureless Euler equations and the isentropic compressible Navier-Stokes equations where the coupling of two systems is through the drag force.
This  coupled system can be  derived, in the hydrodynamic limit, from the particle-fluid equations that are frequently used to study the medical sprays, aerosols and sedimentation problems. For the proposed system, we first construct the local-in-time  classical solutions in an appropriate $L^2$ Sobolev space. We also establish the \emph{a priori} large-time behavior estimate by constructing a Lyapunov functional measuring the fluctuation of momentum and mass from the averaged quantities, and using this together with the bootstrapping argument, we obtain the global classical solution.
The large-time behavior estimate asserts that the  velocity functions of the pressureless Euler and the compressible Navier-Stokes equations are aligned exponentially fast as time tends to infinity. 
\end{abstract}

\maketitle \centerline{\date}

\tableofcontents
%
%
%
%
\section{Introduction}
In this paper, we prove the global existence of classical solutions and obtain their large-time behavior for the coupled hydrodynamic system consisting of the pressureless Euler equations and the isentropic compressible Navier-Stokes equations where the coupling is through the drag force. Specifically, 
 the hydrodynamic system is given by
\begin{align}\label{two-hydro-eqns}
\begin{aligned}
&\pa_t \rho + \nabla_x \cdot (\rho u) = 0,\qquad (x,t) \in \T^3 \times \R_+,\cr
&\pa_t(\rho u) + \nabla_x \cdot (\rho u \otimes u) = - \rho (u - v),\cr 
&\pa_t \na + \nabla_x \cdot (\na v) = 0, \cr
&\pa_t(\na v) + \nabla_x \cdot (\na v \otimes v) + \nabla_x p(\na) + Lv = \rho(u-v),
\end{aligned}
\end{align}
where the pressure $p$ and the Lam\'e operator $L$ are given by
\begin{align*}
\begin{aligned}
& p(\na) = \na^\gamma \quad \mbox{with} \quad \gamma  > 1,\cr
& Lv = -\mu \Delta_x v - (\mu + \lambda)\nabla_x ( \nabla_x \cdot v ) \quad \mbox{with} \quad \mu > 0 \quad \mbox{and} \quad \lambda + 2\mu > 0.
\end{aligned}
\end{align*}
Here $t\ge0$ is time, $x\in\T^3$ is the spatial coordinate in the three dimensional periodic domain $\T^3$, $\rho(x,t)$ and $\na(x,t)$ represent the particle density and the fluid density at a domain $(x,t) \in \T^3 \times \R_+$, and $u(x,t)$ and $v(x,t)$ represent the corresponding bulk velocities for $\rho(x,t)$ and $\na(x,t)$, respectively.

Equations \eqref{two-hydro-eqns} can be formally derived from the particle-fluid equations, called the Vlasov/compressible Navier-Stokes equations, under the assumption that the particle distribution is mono-kinetic. More precisely, let $f(x,\xi,t)$ be the distribution function of particles at the position-velocity $(x,\xi) \in \T^3 \times \R^3$ at time $t$, and $n$ and $v$ be the fluid density and velocity, respectively. 
Then the motion of the particles and fluid can be described by the following kinetic-fluid equations:
\begin{align}\label{kin-flu}
\begin{aligned}
&\pa_t f + \xi \cdot \nabla_x f + \nabla_\xi \cdot((v - \xi)f) = 0, \quad (x,\xi,t) \in \T^3 \times \R^3 \times \R_+,\cr
&\pa_t \na + \nabla_x \cdot (\na v) = 0, \quad (x,t) \in \T^3 \times \R_+,\cr
&\pa_t (\na v) + \nabla_x \cdot (\na v \otimes v) + \nabla_x p(\na) + Lv = -\int_{\R^3} (v-\xi)f\,d\xi.
\end{aligned}
\end{align}

The kinetic-fluid models describing the interactions between particles and fluids have received considerable attentions due to their medical and engineering applications \cite{BBJM,BDM,VASG,Will}. 
%
%
%
%
%
%
For the modeling and physical backgrounds for the particle-fluid equations, we refer the readers to \cite{Rourke, RM}. For the system \eqref{kin-flu} with a global alignment force, the first author and his collaborators showed the global existence of strong solutions and also obtained the large-time behavior estimates under suitable conditions on the initial data and viscosity $\mu$ in \cite{BCHK3}. For the interactions with incompressible fluids, the global well-posedness and the large-time behavior of solutions are
studied in \cite{BCHK, BCHK2, Choi, CK2, CL}. When the diffusion effect $\Delta_\xi f$ is considered in the Vlasov equations $\eqref{kin-flu}_1$, the system \eqref{kin-flu} is called Vlasov-Fokker-Planck/compressible Navier-Stokes equations. For this system, the global existence of weak solutions is studied in a bounded domain with the Dirichlet or reflection boundary conditions in \cite{MV}, and global existence of the classical solution in the periodic spatial domain is discussed in \cite{CKL}. In \cite{DL}, the existence of global strong solutions and large-time behavior for the Vlasov-Fokker-Planck/compressible Euler equations have been studied in both the whole space and periodic spatial domain.

Here we shall give a brief outline for the derivation of the system \eqref{two-hydro-eqns} from \eqref{kin-flu}. To this end, the macroscopic variables of the local mass $\rho$ and momentum $\rho u$ for the distribution function $f$ are introduced as follows.
\[
\rho(x,t) := \int_{\R^3} f(x,\xi,t)\,d\xi \quad \mbox{and} \quad (\rho u)(x,t):=\int_{\R^3} \xi f(x,\xi,t)\,d\xi \quad \mbox{for} \quad (x,t) \in \T^3 \times \R_+.
\]
We next set the energy-flux $\hat q$, the pressure tensor $\hat\sigma$, and the temperature $\theta$ given by the fluctuation terms:
\begin{align*}
\begin{aligned}
\hat q(x,t) &:= \frac12\int_{\R^3} |\xi - u(x,t)|^2 (\xi - u(x,t)) f(x,\xi,t)\, d\xi, \quad \hat\sigma(x,t) := \int_{\R^3} (\xi - u(x,t))\otimes (\xi - u(x,t)) f(x,\xi,t)\, d\xi,
\end{aligned}
\end{align*}
and
\[
(\rho\theta)(x,t):= \frac12\int_{\R^3} |\xi - u(x,t)|^2 f(x,\xi,t)\,d\xi.
\]
First, by integrating the equation \eqref{kin-flu} in $\xi$, we obtain  the continuity equation:
\[
\frac{d\rho}{dt} + \nabla_x \cdot(\rho u) = 0.
\]
For the momentum equation, we multiply $\eqref{kin-flu}_1$ by $\xi$ and again integrating in $\xi$ to deduce that
\[
\frac{d(\rho u)}{dt}  + \nabla_x \cdot (\rho u \otimes u) + \nabla_x \cdot \hat\sigma = -\rho(u-v).
\]
Then, we multiply the equations $\eqref{kin-flu}_1$ by $\frac{|\xi|^2}{2}$ and integrate in $\xi$ to obtain
\begin{align}\label{f-ener-func1}
\begin{aligned}
\frac12\frac{d}{dt}\int_{\R^3} |\xi|^2 f\,d\xi &= -\frac12\int_{\R^3} |\xi|^2 \lt( \nabla_x \cdot (\xi f) + \nabla_\xi \cdot ((v - \xi)f) \rt) d\xi =:I_1 + I_2,
\end{aligned}
\end{align}
where $I_1$ and $I_2$ are given by 
\begin{align}\label{f-ener-func2}
\begin{aligned}
I_1 &= - \frac12\nabla_x \cdot \lt( \int_{\R^3} |\xi - u|^2 \xi f \,d\xi + \rho|u|^2 u + 2\int_{\R^3} (\xi - u)\otimes (\xi - u) u f\, d\xi\rt)= - \nabla_x \cdot \lt( \hat q + \rho(|u|^2 + \theta)u + \hat\sigma u\rt),\cr
I_2 &= \int_{\R^3} \xi \cdot (v - \xi) f\,d\xi = \rho u \cdot v - \int_{\R^3} |\xi|^2 f\,d\xi= \rho u \cdot v - \lt( \rho|u|^2 + 2\rho \theta\rt)= 2\rho \theta - \rho u \cdot(u-v).
\end{aligned}
\end{align}
Then combining \eqref{f-ener-func1} and \eqref{f-ener-func2}, we obtain
\[
\frac{d}{dt}\lt( \rho\lt(\frac{|u|^2}{2} + \theta\rt)\rt) + \nabla_x \cdot \lt( \lt(\rho(|u|^2 + \theta) + \hat\sigma\rt)u + \hat q\rt) = 2\rho\theta - \rho u \cdot (u-v).
\]
Hence we collect all the equations of macroscopic variables and those of the compressible fluid variables $(n,v)$ to obtain
\begin{align}\label{comp-sys}
\begin{aligned}
&\pa_t \rho + \nabla_x \cdot(\rho u) = 0,\quad (x,t) \in \T^3 \times \R_+,\cr
&\pa_t(\rho u) + \nabla_x \cdot (\rho u \otimes u) + \nabla_x \cdot \hat\sigma = -\rho(u-v),\cr
&\pa_t\lt( \rho\lt(\frac{|u|^2}{2} + \theta\rt)\rt) + \nabla_x \cdot \lt( \lt(\rho(|u|^2 + \theta) + \hat\sigma\rt)u + \hat q\rt) = 2\rho\theta - \rho u \cdot (u-v),\cr
&\pa_t \na + \nabla_x \cdot (\na v) = 0,\cr
&\pa_t (\na v) + \nabla_x \cdot (\na v \otimes v) + \nabla_x p(\na) + Lv = -\int_{\R^3} (v-\xi)f\,d\xi.
\end{aligned}
\end{align}
Notice that the system \eqref{comp-sys} is not closed due to the energy-flux $\hat q$. In order to close the system \eqref{comp-sys}, we assume that the fluctuations are negligible and the velocity distribution is mono-kinetic, i.e., $f(x,\xi,t) = \rho(x,t)\delta(\xi - u(x,t))$, where $\delta$ denotes the standard Dirac delta function. Then,  it is straightforward to check that the system \eqref{comp-sys} reduces to our proposed model \eqref{two-hydro-eqns}. We remark that \eqref{two-hydro-eqns} can also be derived from the Vlasov-Boltzmann/compressible Navier-Stokes equations with the strong inelastic collision effect between particles following the similar argument as in \cite{DM}. We also refer the reader to \cite{CCK,CG,Choi2, GJV,GJV2, HKK, MV2} for the hydrodynamic limit from the particle-fluid equations to the coupled hydrodynamic equations and references therein. 

For the sake of simplicity, 
 we shall reformulate system \eqref{two-hydro-eqns} in the perturbation framework.
 Setting $\na(t, x) := 1+n(t, x)$, we rewrite the system  \eqref{two-hydro-eqns} as follows.
\begin{align}\label{two-hydro-eqns-1}
\begin{aligned}
&\pa_t \rho + \nabla_x \cdot (\rho u) = 0,  \qquad (x,t) \in \T^3 \times \R_+,\cr
&\pa_t(\rho u) + \nabla_x \cdot (\rho u \otimes u) = - \rho (u - v),
\cr
&\pa_t n + \nabla_x \cdot ((n+1)v) = 0, 
\cr
&\pa_t((n+1)v) + \nabla_x \cdot ((n+1)v \otimes v) + \nabla_x p(n+1) + Lv = \rho(u-v) 
\end{aligned}
\end{align}
with initial data:
\bq\label{ini-two-hydro-eqns}
(\rho(x,t),u(x,t),n(x,t),v(x,t))|_{t = 0} = (\rho_0(x), u_0(x), n_0(x), v_0(x)).
\eq
Note that the initial density $n_0$ of compressible flow  has zero mass, i.e., $\int_{\T^3} n_0(x)\, dx = 0$. 

Before we state our main result, we briefly review some relevant existence theory and results  of the large-time behaviors for the  pressureless Euler equations and  compressible Navier-Stokes equations. In the absent of the interactions between the fluids, i.e., no drag force term, we  have the pressureless Euler equations from $\eqref{two-hydro-eqns-1}_1$-$\eqref{two-hydro-eqns-1}_2$:
\begin{align}\label{pre-Euler}
\begin{aligned}
&\pa_t \rho + \nabla_x \cdot (\rho u) = 0,  \cr
&\pa_t(\rho u) + \nabla_x \cdot (\rho u \otimes u) = 0.
\end{aligned}
\end{align}
The pressureless Euler equations \eqref{pre-Euler} have been studied in \cite{KPS, SZ, VDFN, Zeld} to account for the formation of the large scale structures in the universe, and  it has also been used to describe the motion of free particles which stick each other upon the collision \cite{BG,WRS}. One of the main difficulties in analyzing the system \eqref{pre-Euler} arises from the formation of singularities. Specifically,   no matter how smooth the initial data are, the equations may develop a  singularity such as a $\delta$-shock in finite-time. For this reason, it is natural to extend the notion of solutions to the measure-valued solutions. The  existence of measure solutions of Riemann problem is first investigated in \cite{Bou} for the one-dimensional case, and the global existence and the behavior of entropic weak solutions to the system \eqref{pre-Euler} are obtained in \cite{BG,WRS}. We refer the reader to \cite{Chen} and the references therein for general survey of the Euler equations.
Similarly, by neglecting the drag force term  in $\eqref{two-hydro-eqns-1}_3$-$\eqref{two-hydro-eqns-1}_4$, one has the compressible Navier-Stokes equations:
\begin{align}\label{comp-NS}
\begin{aligned}
&\pa_t n + \nabla_x \cdot ((n+1)v) = 0, \cr
&\pa_t((n+1)v) + \nabla_x \cdot ((n+1)v \otimes v) + \nabla_x p(n+1) + Lv = 0.
\end{aligned}
\end{align}
The well-posedness for  \eqref{comp-NS} have been extensively studied in \cite{Cho,CCK1,CK1,Dan, Des,F-N-P,Hoff,Itaya,Lions,MN1,MN2, Ponce}. In particular, Lions provided  general results for weak solutions to the multidimensional compressible Navier-Stokes equations with large initial data \cite{Lions}, and later the existence and regularity of  weak solutions are further studied in \cite{Des,F-N-P}. The local existence of strong solutions is investigated in \cite{CCK1,CK1}, and global existence of classical solutions is studied in \cite{Dan, MN1,MN2} when the initial density $n_0 +1$ is bounded away from zero. For the large-time behaviors of the solutions, we refer the reader to \cite{MN1, MN2, Ponce} for the whole space, and to \cite{KK1, KK2, KS} for the half space or the exterior domain. More recently, the exponential decay of weak solutions in bounded domain is obtained in \cite{F-Z-Z} when the density $n$ has an upper bound.

In the present work, we show the global existence of  classical solutions, and obtain their large-time behavior, which asserts that the fluid velocities are aligned exponentially fast. As mentioned above,
  the pressureless Euler equations \eqref{pre-Euler} may  develop the $\delta$-shock in finite-time even with smooth initial data. Concerning this issue, an interesting question 
is whether the drag force coupling two systems can prevent the formation of 
the finite-time singularities, and whether the system can admit the global classical solutions.
Another natural question is, if the global solution exists, how the solutions behave as time tends to infinity. 
To this end, we employ a Lyapunov functional measuring the fluctuation of momentum and mass from the corresponding averaged quantities:
\begin{equation}\label{Lyap}
\mathcal{L}(t):= \int_{\T^3} \rho |u-m_c|^2 dx + \int_{\T^3} (n+1)| v - j_c|^2 dx + |m_c - j_c|^2 + \int_{\T^3} n^2 dx, 
\end{equation}
where $m := \rho u$, $j := (n+1)v$, 
\begin{equation}\label{def1}
m_c(t) := \frac{\int_{\T^3} \rho u\, dx}{\int_{\T^3} \rho \,dx}, \quad j_c(t) := \int_{\T^3} (n+1) v \,dx, \quad \mbox{and} \quad \rho_c(t) := \int_{\T^3} \rho \, dx.
\end{equation}
By energy estimates, using a dissipative structure of the Navier-Stokes equations and the drag force term coupling two systems, we establish the large-time estimate in Proposition \ref{prop:large-t}, under an appropriate smallness condition, that
\[
\mathcal{L}(t) \leq C_1\mathcal{L}_0e^{-\lambda t} \quad  \text{ for some constants } C_1, \lambda>0.
\]
One of the most important features in establishing the estimate is to derive the dissipative terms using the Bogovskii's argument in the setting of spatial periodic domain. 
It is interesting that this asymptotic behavior estimate (\emph{a priori}) plays an important role in constructing the time-global solutions. We shall give a brief outline of the procedure here.
The asymptotic behavior estimate implies that 
\begin{equation}\label{rlb11}
\int_{\T^3} \rho |u-m_c|^2 dx \le C_1  \mathcal{L}_0e^{-\lambda t}.
\end{equation}
On the other hand, using the characteristic method together with the smallness assumptions $\| \nabla_x u \|_{L^\infty(\T^3 \times (0,T))}$ $\le \epsilon_1$, we obtain the \emph{a priori} lower bound of $\rho(x,t)$, i.e., 
\begin{equation}\label{rlb2} 
\rho(x,t) \ge 
\lt(\min_{x\in\T^3} \rho_0(x)\rt) \exp \lt(- \int_0^t \| \nabla_x u (\cdot,\tau ) \|_{L^\infty( \T^3)} d\tau\rt) \ge \delta_0 e^{-\epsilon_1t}.
\end{equation}
The estimates \eqref{rlb11} and \eqref{rlb2} imply the $L^2$ exponential decay, $\| u(\cdot,t) - m_c(t) \|^2_{L^2(\T^3)} \lesssim e^{-(\lambda-\epsilon_1)t}$, where $\lambda-\epsilon_1>0$. Moreover, upon the interpolation of this and an appropriate high order $L^2$-Sobolev norm of $u$, say $\| u \|_{H^{s+2}(\T^3)}$, we have the exponential decay for $\|\nabla_x u(\cdot,t)\|_{H^s(\T^3)}\lesssim e^{-ct}$ for some $c>0$. This together with Sobolev embedding and \eqref{rlb2} implies the density function $\rho(x,t)$ has  a uniform lower bound, i.e., there exists a positive constant  $\bar \rho$ such that $\rho(x,t)\ge \bar \rho$ for all $x\in\T^3$ and $t\ge0$ 
(see Lemma \ref{lem:rn}
 and Corollary \ref{low-bd-rho}). 
From this, we note that a combination of the drag force term and the dissipative structure of the Navier-Stokes equations can prevent the formation of the finite-time singularities, at least for the variables for the pressureless Euler system. We can reinterpret the drag forcing term as the ``relative damping'', through which stabilizing effect of the Navier-Stokes is transferred to the pressureless Euler system. 
A combination of the large-time estimate and standard energy estimates of the Navier-Stokes equations enables us to obtain the {\it a priori} uniform bounds for the coupled system. Using this, we prove the global existence of classical solutions and justify the time-asymptotic  alignment behavior of \eqref{two-hydro-eqns-1}.

Here we introduce several notations used throughout the paper. For a function $f(x)$, $\|f\|_{L^p}$ denotes the usual $L^p(\T^3)$-norm. $f \ls g$ represents that there exists a positive constant $C>0$ such that $f \leq C g$. We also denote by $C$ a generic positive constant depending only on the norms of the data, but independent of $T$. For simplicity, we often drop $x$-dependence of differential operators $\nabla_x$, that is, $\nabla f := \nabla_x f$ and $\Delta f := \Delta_x f$. For any nonnegative integer $s$, $H^s$ denotes the $s$-th order $L^2$ Sobolev space. $\mc^s([0,T];E)$ is the set of $s$-times continuously differentiable functions from an interval $[0,T]\subset \R$ into a Banach space $E$, and $L^p(0,T;E)$ is the set of the $L^p$ functions from an interval $(0,T)$ to a Banach space $E$. $\nabla^s$ denotes any partial derivative $\pa^\alpha$ with multi-index $\alpha, |\alpha| = s$. 
Before we state our main result, we define the solution space:
\begin{align*}
\begin{aligned}
\mathcal{I}_s(T):= \big\{ (\rho,u,n,v): &\, \rho \in \mc([0,T];H^s(\T^3)) \cap\mc^1([0,T];H^{s-1}(\T^3)) ,
\cr 
&  u\in \mc([0,T];H^{s+2}(\T^3)) \cap \mc^1([0,T];H^{s+1}(\T^3)),  \cr
& n \in \mc([0,T];H^{s+1}(\T^3)) \cap\mc^1([0,T];H^{s}(\T^3)),\quad \mbox{and} \cr
& v\in \mc([0,T];H^{s+1}(\T^3))\cap L^2(0,T;H^{s+2}(\T^3)) \big\}.
\end{aligned}
\end{align*}
\begin{theorem}\label{thm:main} Let $s > \frac52$. Suppose that the initial data $(\rho_0,u_0,n_0,v_0)$ satisfy the following conditions:
\begin{align*}
\begin{aligned}
& (i)\,\,\,\,\,\,  \min_{x \in \T^3} \rho_0(x) > 0, \quad 1 + \min_{x\in\T^3} n_0(x) > 0,\cr
& (ii) \,\,\,\,(\rho_0,u_0,n_0, v_0) \in H^s(\T^3) \times H^{s+2}(\T^3) \times H^{s+1}(\T^3) \times H^{s+1}(\T^3).
\end{aligned}
\end{align*}
If $\|\rho_0\|_{H^s} + \|u_0\|_{H^{s+2}} + \|n_0\|_{H^{s+1}} + \|v_0\|_{H^{s+1}} \leq \varepsilon_1$ for sufficiently small $\varepsilon_1>0$, the initial value problem \eqref{two-hydro-eqns-1}-\eqref{ini-two-hydro-eqns}  has a unique global classical solution $(\rho,u,n,v) \in \mathcal{I}_s(\infty)$ satisfying
\begin{equation}\label{large-behav}
\mathcal{L}(t) \leq C\mathcal{L}_0e^{-\lambda t}, \quad t \geq 0
\end{equation}
for some positive constants $\lambda$ and  $C > 0$ independent on $t$, where $\mathcal{L}(t)$ is defined in \eqref{Lyap}.
\end{theorem}

\begin{remark} 
 We note that \eqref{large-behav} implies that the velocity functions, $u(x,t)$ and $v(x,t)$ are aligned, and converge to $\frac{1}{1 + \rho_c(0)}(m_c(0) + j_c(0))$ in $L^\infty(\T^3)$ as $t\to\infty$. To see this, we first find that $\rho(x,t)$ has a uniform bound from below, i.e., $\rho(x,t)\ge \bar \rho$ for some constant $\bar \rho>0$(see Corollary \ref{low-bd-rho}), and this together with \eqref{large-behav} implies $\| u - m_c\|^2_{L^2} \lesssim e^{-\lambda t}$. By interpolating $L^2(\T^3)$ and $H^3(\T^3)$ norms, and using Sobolev inequality, we deduce that $\| u - m_c \|_{L^\infty} \lesssim \| u - m_c \|_{H^2} \lesssim e^{-\tilde\lambda t}$ for some $\tilde\lambda > 0$.
Similarly we also obtain that $\| v - j_c \|_{L^\infty} \lesssim e^{-\tilde\lambda t}.$ On the other hand, it follows from the conservation of the total momentum(see Lemma \ref{lem:energy}) that 
\[
\rho_c(0)m_c(t) + j_c(t) = \rho_c(0)m_c(0) + j_c(0)\quad \mbox{for} \quad t \geq 0,
\]
and this together with \eqref{large-behav} implies that
\[
(1 + \rho_c(0))\lt| m_c(t) - \frac{1}{1 + \rho_c(0)}(m_c(0) + j_c(0))\rt| \to 0 \quad \mbox{as} \quad t \to \infty.
\]
Moreover, it is an immediate consequence from \eqref{large-behav} that $| m_c(t) - j_c(t) | \to 0$. Combining these, one can conclude that 
$$\lt\| u(\cdot,t) - \frac{1}{1 + \rho_c(0)}(m_c(0) + j_c(0)) \rt\|_{L^\infty} + \lt\| v(\cdot,t) -\frac{1}{1 + \rho_c(0)}(m_c(0) + j_c(0)) \rt\|_{L^\infty} \to 0 \quad \text{ as } \quad t\to\infty.$$
That is, the two fluid velocities converge to the averaged initial total momentum in $L^\infty(\T^3)$ as time evolves exponentially fast.


%
%
%
%
\end{remark}

The rest of this paper is organized as follows.
 In Section \ref{sec3}, we study the local existence of the unique classical solutions to  system \eqref{two-hydro-eqns-1}. For this, we use the fact that the equations $\eqref{two-hydro-eqns-1}_3$-$\eqref{two-hydro-eqns-1}_4$ have the structure of symmetric hyperbolic system. We linearize the system, and show the existence and the uniform boundedness of the solutions for the linearized system. Then, we construct the approximated solutions, and provide that they are Cauchy sequences in the proposed Sobolev spaces. Section \ref{sec4} is devoted to discuss the {\it a priori} estimate for the large-time behavior of solutions which actually gives us the uniform bounds of the density of pressureless Euler equations. Finally, in Section \ref{sec5}, we provide the {\it a priori} estimates of solutions in the proposed Sobolev spaces in Theorem \ref{thm:main} with the aid of the large-time behavior estimate. This concludes that the local solutions can be extended to the global solution and that the large-time behavior estimates are justified for the global solutions.

\section{Local existence of classical solutions}\label{sec3}
In this section, we discuss local existence of the unique classical solution. 
For this, we  shall use the structure of symmetric hyperbolic system for $\eqref{two-hydro-eqns-1}_3$-$\eqref{two-hydro-eqns-1}_4$. Our system \eqref{two-hydro-eqns-1} can be rewritten as 
\begin{align}\label{local-sys}
\begin{aligned}
&\pa_t \rho + \nabla \cdot (\rho u) = 0, \quad (x,t) \in \T^3 \times \R_+,\cr
&\pa_t (\rho u) + \nabla \cdot (\rho u \otimes u) = -\rho(u-v),\cr
&A^0(w)\pa_t w + \sum_{j=1}^3 A^j(w)\pa_j w = A^0(w)E_1(n,v) + A^0(w)E_2(\rho,u,n,v),
\end{aligned}
\end{align}
where $w = (n,v)^T$, 
\[
A^0(w) =
\left( \begin{array}{cc}
\gamma (1+n)^{\gamma-2} & 0 \\
0 & (1+n) \mathbb{I}_{3 \times 3} \\
\end{array} \right),
\]
\[
A^j(w) = A^0(w)
\left( \begin{array}{cccc}
v_j & (1 + n)\delta_{1j} & (1 + n)\delta_{2j} & (1 + n)\delta_{3j}\\
\gamma (1 + n)^{\gamma-2}\delta_{1j} & v_j & 0 & 0 \\
\gamma (1 + n)^{\gamma-2}\delta_{2j} & 0 & v_j & 0 \\
\gamma (1 + n)^{\gamma-2}\delta_{3j} & 0 & 0 & v_j \\
\end{array} \right), 
\]
\[
E_1(n,v) =\frac{1}{1+n}
\left( \begin{array}{cccc}
0 & 0 & 0 & 0\\
0 & (\displaystyle -Lv)_1 & 0 & 0 \\
0 & 0 & (\displaystyle -Lv)_2 & 0  \\
0 & 0 & 0 & (\displaystyle -Lv)_3  \\
\end{array} \right),
\]
and
\[
E_2(\rho,u,n,v) = \frac{1}{1+n}
\left( \begin{array}{cccc}
0 & 0 & 0 & 0 \\
0 & \displaystyle \rho(u_1-v_1) & 0 & 0 \\
0 & 0 & \displaystyle \rho(u_2-v_2) & 0  \\
0 & 0 & 0 & \displaystyle \rho(u_3-v_3) \\
\end{array} \right).
\]
%
%
\newcommand{\scon}{s>\frac52}
Now we present the local existence result.
\begin{theorem}\label{thm:local}
Let $\scon$. For any positive constants $\epsilon_0< M_0$, 
there is a positive constant $T_0$ depending only on $\epsilon_0$ and $M_0$ such that if $\|\rho_0\|_{H^s} + \|u_0\|_{H^{s+2}} + \|n_0\|_{H^{s+1}} + \|v_0\|_{H^{s+1}} \leq \epsilon_0$, then the Cauchy problem \eqref{local-sys} has a unique classical solution $(\rho,u,n,v)\in \mi_s(T_0)$   satisfying
\[
\sup_{0 \leq t \leq T_0}\lt(\|\rho(t)\|_{H^s} + \|u(t)\|_{H^{s+2}} + \|n(t)\|_{H^{s+1}} + \|v(t)\|_{H^{s+1}}\rt) \leq M_0.
\]
\end{theorem}
In what follows, we shall give a brief outline of the proof for local existence. 
%
%
%
%
\subsection{Solvability of the associated linear system} As a first step, we shall study the existence and uniqueness of the classical solutions for the  linearized system associated with \eqref{local-sys}. 

We set the norm of the function $W:=(\rho, u, n, v) \in \mi(T;s)$ as
\[
\|W\|_{\mathcal{N}^s} := \|\rho\|_{H^s} + \|u\|_{H^{s+2}} + \|n\|_{H^{s+1}} + \|v\|_{H^{s+1}}.
\]
For a given $\overline W := (\bar\rho,\bar u, \bar n, \bar v) \in \mathcal{I}(T;s)$, we consider the associated linear system:
\begin{align}\label{li-sys-a}
\begin{aligned}
&\pa_t \rho + \bar u \cdot \nabla \rho + \rho \nabla \cdot \bar u = 0,\cr
&\rho \pa_t u + \rho \bar u \cdot \nabla u = - \rho (\bar u - \bar v),\cr
&A^0(\bar w) \pa_t w + \sum_{j=1}^3 A^j(\bar w)\pa_j w = A^0(\bar w)E_1(\bar n, v) + A^0(\bar w)E_2(\bar \rho, \bar u, \bar n, \bar v),
\end{aligned}
\end{align}
with the initial data $W_0\in {\mn^s}$, where $\bar w = (\bar n, \bar v)^T$. Then, one can show that the system \eqref{li-sys-a}  has a unique solution  $W \in \mathcal{I}(T;s)$. Specifically we have the following lemma.
\begin{lemma}\label{lem:ext-linear1} Let $\scon$. Suppose that $\overline W \in \mathcal{I}(T ;s)$ and $W_0\in {\mn^s}$. Then the initial value problem \eqref{li-sys-a} with initial data $W_0\in {\mn^s}$ has a unique solution $W\in \mathcal{I}(T;s)$.
\end{lemma}
\begin{proof}
This can be proven by a standard linear theory for the transport equations and hyperbolic system with an appropriate  modification of regularity. 
\end{proof}
%
%
%
%
Now, we  construct the approximation sequence $W^m = (\rho^m,u^m,n^m,v^m)$ for the system \eqref{local-sys} by solving the linear system:
\begin{align}\label{approx-sys}
\begin{aligned}
&\pa_t \rho^{m+1} + u^m \cdot \nabla \rho^{m+1} + \rho^{m+1}\nabla \cdot u^m = 0,\cr
&\rho^{m+1}\pa_t u^{m+1} + \rho^{m+1}u^m \cdot \nabla u^{m+1} = -\rho^{m+1}(u^m - v^m),\cr
&A^0(w^m) \pa_t w^{m+1} + \sum_{j=1}^3 A^j(w^m)\pa_j w^{m+1} = A^0(w^m)E_1(n^m,v^{m+1}) + A^0(w^m)E_2(W^m),
\end{aligned}
\end{align}
with the initial data and first iteration step defined by
\[
W^m(x)|_{t=0} = W_0 \quad \mbox{for all} \quad n \geq 1, x \in \T^3,
\]
and
\[
W^0(x,t) = W_0, \quad (x,t) \in \T^3 \times \R_+.
\]
Here, for notational convenience, we set $w^m = (n^m,v^m)$. Then, by Lemma \ref{lem:ext-linear1}, one has that the approximation sequence $\{W^m\}_{m=0}^\infty$ is well-defined. Furthermore, by a standard energy method for the transport equations and hyperbolic system, we can obtain the uniform bound for $\{W^m\}_{m=0}^\infty$ with a suitable choice of time $T_0$ as follow.
\begin{lemma}\label{prop:invar} Let $\scon$. For any positive constants $\epsilon_0< M_0$, there exists $T_0>0$ such that if $\|W_0\|_{\mn^s} \leq \epsilon_0$, then for each $m \geq 0$, $W^m \in \mi_s(T_0)$ is well-defined and
\[
\sup_{0 \leq t \leq T_0}\|W^m(t)\|_{\mn^s} \leq M_0 \quad \mbox{for all} \quad m \geq 0.
\]
\end{lemma}
\begin{proof}
For the detailed proof, see Appendix \ref{Local}.
\end{proof}

Next one can show, by establishing the estimate for the difference $(W^{m+1} -W^m)$ in a standard way, that the approximation sequence $\{W^m\}_{m=0}^\infty$ is a Cauchy sequence in $\mc([0,T_0];L^2(\T^3)) \times \mc([0,T_0];H^1(\T^3)) \times \mc([0,T_0];L^2(\T^3)) \times (\mc([0,T_0];L^2(\T^3)) \cap L^2(0,T_0;H^1(\T^3)))$.
\begin{lemma}\label{lem:cauchy} Let $(\rho^m,u^m,n^m,v^m)$ be a sequence of the approximated solutions with the initial data $(\rho_0,u_0,n_0,v_0)$ satisfying $\|W_0\|_{\mn^s} \leq \epsilon_0$, where $\epsilon_0$ chosen in Proposition \ref{prop:invar}. Then $(\rho^m,u^m,n^m,v^m)$ is a Cauchy sequence in $\mc([0,T_0];L^2(\T^3)) \times \mc([0,T_0];H^1(\T^3)) \times \mc([0,T_0];L^2(\T^3)) \times (\mc([0,T_0];L^2(\T^3)) \cap L^2(0,T_0;H^1(\T^3)))$.
\end{lemma}
\begin{proof}
For the detailed proof, see Appendix \ref{Local}.
\end{proof}
Interpolating this with the uniform bound of $\{W^m\}_{m=0}^\infty$ in $\mc([0,T_0];H^s(\T^3)) \times \mc([0,T_0];H^{s+2}(\T^3)) \times \mc([0,T_0];H^{s+1}(\T^3)) \times \mc([0,T_0];H^{s+1}(\T^3))$ yields that
\[
(\rho^m,u^m) \to (\rho,u) \quad \mbox{in } \mc([0,T_0];H^{s-1}(\T^3)) \times \mc([0,T_0];H^{s+1}(\T^3)) \quad \mbox{as } m \to \infty,
\]
and
\[
(n^m,v^m) \to (n,v) \quad \mbox{in }\mc([0,T_0];H^{s}(\T^3)) \times \mc([0,T_0];H^{s}(\T^3)) \quad \mbox{as } m \to \infty.
\]
Thus it only remains to show that $(\rho,u,n,v) \in \mi_s(T_0)$.
To this end, one can first prove the time-right continuity using a standard functional analytic argument together with the revisited energy estimates. For the hyperbolic variables $(\rho, n, v)$, by simply considering the time reversal problem, one can show that $(\rho, n,v)$ is the time-left continuous, in turn, it has the desired regularity. 
However, since the momentum equations are not time-reversible, one has to treat the compressible fluid velocity $v$ in a different way. For this, we obtain a better energy estimate thanks to the smoothing effect of diffusion for $v$ to show the desired regularity \cite{MB}. Here we  remark that the energy method with the time-translated mollifier technique developed in \cite{KST} specialized for the initial-boundary value hyperbolic problem can also be used to show the desired regularity without the time-reversal argument. 

\textbf{Uniqueness:} Let $(\rho_1,u_1,n_1,v_1)$ and $(\rho_2,u_2,n_2,v_2)$ be the classical solutions obtained in the part of existence with the same initial data $(\rho_0,u_0,n_0,v_0)$. We set $\Delta(t)$ a difference between two classical solutions:
\[
\Delta(t) := \|\rho_1(t) - \rho_2(t)\|_{L^2}^2 + \|u_1(t) - u_2(t)\|_{H^1}^2 + \|n_1(t) - n_2(t)\|_{L^2}^2 + \|v_1(t) - v_2(t)\|_{L^2}^2.
\]
Then it directly follows from Lemma \ref{lem:cauchy} that
\[
\Delta(t) \lesssim\int_0^t \Delta(s)\,ds,
\]
with $\Delta(0) = 0$. This yields that $\Delta(t) \equiv  0$ for all $t \in [0,T_0]$ and 
\begin{align*}
\begin{aligned}
&\rho_1 \equiv \rho_2 \quad \mbox{in } \mc([0,T_0];L^2(\T^3)), \quad u_1 \equiv u_2 \quad \mbox{in } \mc([0,T_0];H^1(\T^3)),\cr
&n_1 \equiv n_2 \quad \mbox{in } \mc([0,T_0];L^2(\T^3)), \mbox{ and } v_1 \equiv v_2 \quad \mbox{in } \mc([0,T_0];L^2(\T^3))\cap L^2(0,T_0;H^1(\T^3)).
\end{aligned}
\end{align*}
This concludes the uniqueness of classical solutions
%
%
%
%
\section{A priori estimates for the large time behavior}\label{sec4}
In this section, we study the large time behavior of the classical solutions to system \eqref{two-hydro-eqns-1}-\eqref{ini-two-hydro-eqns}. The estimates for the large time behavior will be crucially used to get the uniform bound for $\rho$ in $H^s$-norm (see Lemma \ref{lem:rn}), by which one can conclude that the finite-time blow-up of the density for the pressureless Euler equations cannot occur. Before we proceed, we define 
\begin{equation} E(t) := \int_{\T^3} \rho|u|^2 + (n+1)|v|^2 dx + \frac{1}{\gamma-1}\int_{\T^3} (n+1)^\gamma dx.
\end{equation}

\begin{proposition}\label{prop:large-t}Let $(\rho,u,n,v)$ be the classical solutions to  \eqref{two-hydro-eqns-1}-\eqref{ini-two-hydro-eqns} in the interval $[0,T]$ satisfying 
\begin{enumerate}
\item $\rho \in [0,\bar\rho]$, $n +1 \in [0, \bar n]$, $\rho_c(0) \in (0,\infty)$,
\item $v \in L^\infty(0,T;L^\infty(\T^3))$,
\item $E(0)$ is small enough.
\end{enumerate}
Then we have
\[
\mathcal{L}(t) \leq C_1\mathcal{L}_0e^{-\lambda t}, \quad t \in [0,T],
\]
for some constants $C_1$ and $\lambda > 0$, where 
\[
\mathcal{L}(t)= \int_{\T^3} \rho |u-m_c|^2 dx + \int_{\T^3} (n+1)| v - j_c|^2 dx + |m_c - j_c|^2 + \int_{\T^3} n^2 dx.
\]
\end{proposition}
\begin{remark}
Note that no assumption on the lower bounds of $\rho$ and $n+1$ has been made for Proposition \ref{prop:large-t}. As mentioned before, we make use of the Lyapunov functional proposed in \cite{CK} for the compressible fluid. 
This gives a sharper result than the one obtained in \cite{BCHK3}. 
\end{remark}

\begin{lemma}\label{lem:energy} Let $(\rho,u,n,v)$ be the global classical solutions to the system \eqref{two-hydro-eqns-1}-\eqref{ini-two-hydro-eqns}. Then there hold
\begin{align*}
\begin{aligned}
&(i) \,\text{Conservations of the masses and total momentum:}\cr
& \qquad \qquad \frac{d}{dt}\int_{\T^3} \rho \,dx = \frac{d}{dt}\int_{\T^3} n \,dx = 0 \quad \mbox{and} \quad \frac{d}{dt}\int_{\T^3} (\rho u + (n+1)v) \,dx = 0.\cr
&(ii)\,\text{Dissipation of the total energy:}\cr
&\qquad  \qquad \frac12\frac{d}{dt}E(t)+ \mu\int_{\T^3} |\nabla v|^2 dx + (\mu + \lambda)\int_{\T^3} |\nabla \cdot v|^2 dx +\int_{\T^3} \rho |u-v|^2 dx = 0.
\end{aligned}
\end{align*}
\end{lemma}
\begin{proof}
A straightforward computation yields the conservation of masses and the total momentum. For the estimate of dissipation of the total energy, we use the following relation
\[
\int_{\T^3} v \cdot \nabla p\,dx = \frac{1}{\gamma-1}\frac{d}{dt}\int_{\T^3} p\,dx,
\]
to deduce 
\[
\frac12\frac{d}{dt}\int_{\T^3} \rho|u|^2 dx = - \int_{\T^3} \rho(u-v) \cdot u\,dx,
\]
and
\[
\frac12\frac{d}{dt}\int_{\T^3} (n+1)|v|^2 + \frac{2}{\gamma-1} (n+1)^\gamma dx = -\mu\int_{\T^3} |\nabla v|^2 dx - (\mu + \lambda)\int_{\T^3} |\nabla \cdot v|^2 dx  + \int_{\T^3} \rho (u-v)\cdot v\,dx.
\]
Hence we obtain the desired result by combining the above two equalities.
\end{proof}
\begin{remark}\label{rmk:energy}
It follows from Lemma \ref{lem:energy} that
\[
E(t) \leq E(0) =:E_0 \quad \mbox{for all} \quad t \geq 0.
\]
\end{remark}

%
%
%
%
The following is the Moser inequalities that will be used later. 
\begin{lemma}\label{3.2}
(i) Let $k \in \N, \, p \in [1,\infty], \, h \in \mc^k(\T^3)$. Then there exists a positive constant $c = c(k,p,h)$ such that
\[
\|\nabla^k h(w)\|_{L^p} \leq c\|w\|_{L^\infty}^{k-1}\|\nabla^k w\|_{L^p},
\]
for all $w \in (W^{k,p} \cap L^\infty)(\T^3)$. \newline

(ii) For any pair of functions $f,g \in H^m(\T^3) \cap L^\infty(\T^3)$, we obtain
\[
\|\nabla^k(fg)\|_{L^2} \ls \|f\|_{L^\infty}\|\nabla^k g\|_{L^2} + \|\nabla^k f\|_{L^2}\|g\|_{L^\infty}.
\]
Furthermore if $\nabla f \in L^\infty(\T^3)$, we have
\[
\|\nabla^k(fg) - f\nabla^k g\|_{L^2} \ls \|\nabla f\|_{L^\infty}\|\nabla^{k-1}g\|_{L^2} + \|\nabla^k f\|_{L^2}\|g\|_{L^\infty}.
\]
\end{lemma}
The proof of Lemma \ref{3.2} can be found in section 4 in \cite{Ra}. 
\begin{lemma}\label{lem:asym} Let $(\rho,u,n,v)$ be the classical solutions to  \eqref{two-hydro-eqns-1}-\eqref{ini-two-hydro-eqns}.
Then there hold
\begin{align*}
\begin{aligned}
& (i) \,\,\,\,\,\,
\frac12\frac{d}{dt} \int_{\T^3} \rho|u - m_c|^2 dx = -\int_{\T^3} \rho (u-m_c)\cdot (u-v)dx.\cr
& (ii) \,\,\,\, \frac12\frac{d}{dt}\lt(\int_{\T^3} (n+1) |v - j_c|^2 dx + \frac{2}{\gamma - 1}\int_{\T^3} (n+1)^\gamma dx \rt)\cr
&\qquad \qquad + \mu \int_{\T^3} |\nabla v|^2 dx + (\mu + \lambda)\int_{\T^3} |\nabla \cdot v|^2 dx = \int_{\T^3} \rho (v - j_c) \cdot (u-v) dx.\cr
& (iii) \,\, \frac12\frac{d}{dt}|m_c - j_c|^2 = -\frac{1 + \rho_c(0)}{\rho_c(0)} \int_{\T^3} \rho(m_c - j_c)\cdot (u-v) dx.
\end{aligned}
\end{align*}
\end{lemma}
\begin{proof} A straightforward computation gives  $(i)$ and  $(ii)$. For the estimate of $(iii)$, we use the conservation of total momentum, 
\[
\frac{d}{dt}\int_{\T^3} (\rho + (n+1)v)dx = \rho_c(0)m_c^\prime(t) + j_c^\prime(t) = 0.
\]
This implies that
\begin{align*}
\begin{aligned}
\frac12\frac{d}{dt}|m_c(t) - j_c(t)|^2 &= (m_c(t) - j_c(t)) \cdot (m_c^\prime(t) - j_c^\prime(t))\cr
&=(m_c(t) - j_c(t)) \cdot ((1 + \rho_c(0))m_c^\prime(t))\cr
&=-\frac{1 + \rho_c(0)}{\rho_c(0)} (m_c(t) - j_c(t)) \cdot \int_{\T^3} \rho (u-v) dx.
\end{aligned}
\end{align*}
\end{proof}
Now we define a temporal interacting energy-variation $\mathcal{E}$ and the corresponding dissipation $\mathcal{D}$ as follows.
\begin{align*}
\begin{aligned}
\mathcal{E}(t) &:= \int_{\T^3} \rho |u - m_c|^2 dx + \int_{\T^3} (n+1) | v - j_c|^2 dx + \frac{2}{\gamma - 1}\int_{\T^3} (n+1)^\gamma dx +  \frac{\rho_c}{1 + \rho_c}|m_c - j_c|^2,\cr
\end{aligned}
\end{align*} and
\begin{align*}
\begin{aligned}
\mathcal{D}(t) &:= \mu\int_{\T^3}|\nabla v|^2 dx + (\mu + \lambda)\int_{\T^3}|\nabla \cdot v|^2 dx + \int_{\T^3} \rho |u-v|^2 dx,
\end{aligned}
\end{align*}
where $\rho_c := \rho_c(0)$. Then it follows from Lemma \ref{lem:asym} that 
\[
\frac12\frac{d}{dt}\mathcal{E}(t) + \mathcal{D}(t) = 0. 
\]
We next provide the following elementary estimates for the pressure and local momentum of the compressible fluid.
\begin{lemma}\label{press-lem}
\emph{\cite{F-Z-Z}} 1. Let $r_0, \bar{r} > 0$ and $\gamma > 1$ be
given constants, and set
\[
f(r;r_0) := r\int_{r_0}^{r} \frac{h^{\gamma} - r_0^{\gamma}}{h^2} \,dh,
\]
for $r \in [0,\bar{r}]$. Then, there exist positive constants
$C_1$ and $C_2$ such that
\[
C_1(r_0, \bar{r})(r - r_0)^2 \leq f(r;r_0) \leq C_2(r_0, \bar{r})(r - r_0)^2 \quad \mbox{for all } r \in [0,\bar{r}].
\]
2. There holds 
\[
\frac{1}{\gamma-1}\frac{d}{dt}\int_{\T^3} (n+1)^\gamma dx = \frac{d}{dt}\int_{\T^3} f(n+1;1) dx.
\]
\end{lemma}
\begin{lemma}\label{jc-est-lem} Let $j_c$ be the local momentum  of the compressible fluid defined in \eqref{def1}. Then there holds
\bq\label{est-jc}
|j_c|^2 \leq E_0 \quad \mbox{and} \quad |j_c^\prime|^2 \leq \rho_c\int_{\T^3} \rho|u-v|^2 dx.
\eq
\end{lemma}
\begin{proof}We first obtain by using the dissipation of the total energy in Remark \ref{rmk:energy} that
\[
|j_c|\leq \int_{\T^3} (n+1)|v|dx \leq \lt(\int_{\T^3} (n+1)|v|^2 dx\rt)^\frac12 \leq E(t)^\frac12 \leq E_0^\frac12,
\]
where we used $\int_{\T^3} n \,dx = 0$. We also easily find that
\[
|j_c^\prime|  = \left| \int_{\T^3} \rho(u-v) dx \right| \leq \int_{\T^3} \rho|u-v|dx \leq \rho_c^\frac12\lt(\int_{\T^3} \rho|u-v|^2 dx\rt)^\frac12.
\]
This completes the proof.
\end{proof}
Using Lemma \ref{press-lem}.2, the temporal interacting energy-variation $\mathcal{E}$ can be written as
\begin{align*}
\begin{aligned}
\mathcal{E}(t) &:= \int_{\T^3} \rho |u - m_c|^2 dx + \int_{\T^3} (n+1) | v - j_c|^2 dx + 2\int_{\T^3} (n+1)\int_1^{n+1}\frac{h^\gamma - 1}{h^2}\,dhdx +  \frac{\rho_c}{1 + \rho_c}|m_c - j_c|^2,
\end{aligned}
\end{align*}
and it satisfies 
\[
\frac12 \frac{d}{dt}\mathcal{E}(t) + \mathcal{D}(t) = 0.
\]
Furthermore, we find that $\mathcal{E}$ is equivalent to our proposed Lyapunov functional $\mathcal{L}$ in Proposition \ref{prop:large-t}, i.e., there exists a positive constant $C$ such that
\[
C^{-1} \mathcal{L}(t) \leq \mathcal{E}(t) \leq C \mathcal{L}(t) \quad \mbox{for} \quad t \in [0,T],
\]
due to Lemma \ref{press-lem}-(1).
However, the dissipation $\mathcal{D}$ does not give the desired damping effect for the Lyapunov functional $\mathcal{L}$. More specifically, one can obtain the damping terms from $\mathcal{D}$ except the one for the density of the compressible fluid.
\begin{lemma}\label{lem:dissi}There exists a positive constant $C>0$ such that
\[
\mathcal{L}_{p}(t) \leq C\mathcal{D}(t) \quad \mbox{for} \quad t \in [0,T],
\]
where $\mathcal{L}_p$ is given by
\[
\mathcal{L}_p := \mathcal{L} - \int_{\T^3} n^2 dx = \int_{\T^3} \rho |u-m_c|^2 dx + \int_{\T^3} (n+1)| v - j_c|^2 dx + |m_c - j_c|^2.
\]
\end{lemma}
\begin{remark}
In \cite{Pere}, the compressible Navier-Stokes equations without the pressure term is studied. If we consider the system $\eqref{two-hydro-eqns-1}_3-\eqref{two-hydro-eqns-1}_4$ with no pressure term, i.e., p $\equiv 0$, then it follows from
Lemma \ref{lem:dissi} that we have the exponential alignment between the two fluid velocities by choosing the Lyapunov functional $\mathcal{L}_p$ instead of $\mathcal{L}$. More specifically, if we set
\[
\mathcal{E}_p(t) := \int_{\T^3} \rho |u - m_c|^2 dx + \int_{\T^3} (n+1) | v - j_c|^2 dx +  \frac{\rho_c}{1 + \rho_c}|m_c - j_c|^2,
\]
then it is obvious to get there exists a positive constant $C > 0$ such that
\[
C^{-1}\mathcal{L}_p(t) \leq \mathcal{E}_p(t) \leq C\mathcal{L}_p(t) \quad \mbox{for} \quad t \in [0,T].
\]
Then since $\mathcal{E}_p(t) \leq C\mathcal{L}_p(t) \leq C\mathcal{D}(t)$, we obtain
\[
\frac12\frac{d}{dt}\mathcal{E}(t) + C^{-1}\mathcal{E}(t) \leq 0, \quad \mbox{for} \quad t \in [0,T],
\]
and this yields
\[
\mathcal{L}_p(t) \leq \mathcal{E}(t) \leq \mathcal{E}_0e^{-Ct}, \quad \mbox{for} \quad t \in [0,T].
\]
\end{remark}
\begin{proof}[Proof of Lemma \ref{lem:dissi}] We first find that
\begin{align}\label{lem:equiv1}
\begin{aligned}
\frac12\int_{\T^3} (n+1)|v - j_c|^2 dx &\leq \int_{\T^3} (n+1)|v-v_c|^2 dx + |v_c - j_c|^2 \leq 2\int_{\T^3} (n+1)|v-v_c|^2 dx\leq 2\bar n c_p\int_{\T^3} |\nabla v|^2 dx.
\end{aligned}
\end{align}
Here we have used
\[
|v_c - j_c|^2 = \lt| \int_{\T^3} (n+1)(v - v_c)dx\rt|^2 \leq \int_{\T^3} (n+1)|v - v_c|^2 dx,
\]
where $c_p$ is the constant from Sobolev embedding. We also deduce that 
\begin{align}\label{lem:equiv2}
\begin{aligned}
\frac12\lt( \int_{\T^3} \rho|u - m_c|^2 dx + \rho_c |m_c - j_c|^2\rt) &\leq \int_{\T^3} \rho |u - v|^2 dx + 3\int_{\T^3} \rho |v - j_c|^2 dx\cr
&\leq \int_{\T^3} \rho |u-v|^2 dx + 6c_p\lt( \rho_c \bar{n}+ \bar{\rho}\rt)\int_{\T^3} |\nabla v|^2 dx,
\end{aligned}
\end{align}
where we used
\begin{align*}
\begin{aligned}
\int_{\T^3} \rho|u-v|^2 dx &= \int_{\T^3} \rho|u-m_c + m_c - j_c + j_c - v|^2 dx\cr
&= \int_{\T^3} \rho|u-m_c|^2 dx + \rho_c|m_c - j_c|^2 + \int_{\T^3} \rho|v-j_c|^2 dx\cr
&\quad + 2\int_{\T^3} \rho(m_c - j_c)\cdot (j_c - v) dx + 2\int_{\T^3} \rho (u - m_c) \cdot (j_c - v) dx\cr
&\geq \frac12\int_{\T^3} \rho|u-m_c|^2 dx + \frac{\rho_c}{2}|m_c - j_c|^2 - 3\int_{\T^3} \rho |v - j_c|^2 dx.
\end{aligned}
\end{align*}
Hence we combine the estimates \eqref{lem:equiv1} and \eqref{lem:equiv2} to conclude
\begin{align*}
\begin{aligned}
& \frac12\lt( \int_{\T^3} \rho |u - m_c|^2 dx + \rho_c |m_c - j_c|^2 + \int_{\T^3} (n+1)|v - j_c|^2 dx\rt)\cr
& \qquad \leq \int_{\T^3} \rho |u-v|^2 dx + 2c_p\lt(  3\lt( \rho_c \bar n + \bar \rho \rt) + \bar n\rt)\int_{\T^3} |\nabla v|^2 dx,
\end{aligned}
\end{align*}
that is,
\[
\mathcal{L}_p(t) \leq C\mathcal{D}(t), \quad \mbox{for} \quad t \in [0,T],
\]
where $C$ is a positive constant satisfying
\[
C \geq \frac{2}{\min\{1,\rho_c\}}\max\lt\{ 1, \frac{2c_p\lt(  3\lt( \rho_c \bar n + \bar \rho \rt) + \bar n\rt)}{\mu} \rt\} > 0.
\]
This completes the proof.
\end{proof}
\newcommand{\mb}{\mathcal{B}}
In order to obtain the correct dissipation of $\mathcal{L}$, we present the periodic version of Bogovskii's argument, and it shows that pressure of the compressible fluid gives the desired dissipation. 


For $f\in H^{s-1}(\T^3),s\ge1$ with $\int_{\T^3} f=0$, one can define a linear operator 
$$\mathcal{B} : \left\{ f \in H^{s-1}(\T^3) | \int_{\T^3} f dx = 0
\right\} \mapsto \lt[H^{s}(\T^3)\rt]^3$$ by 
$\mb[f]:= \nabla \phi$ where  $\phi$ is a unique solution to the equation
$\Delta \phi = f$ in $\T^3$ with $\int_{\T^3} \phi  =0$.  
 Moreover there holds 
 \begin{equation}\label{elliptic0}
\| \phi\|_{H^{s+1}} \leq C\|f\|_{H^{s-1}} \quad \mbox{for some positive constant} \quad C > 0.
\end{equation}
This unique solvability and the estimate is true by the standard elliptic regularity theory \cite[Lemma 7.9]{Majda}. 

Then the relations between the norms of $\mb[f]$ and $f$ can be given in the following lemma.
%
%
%
%
%
\begin{lemma}\label{oper-b-lem}
For $f\in L^2(\T^3)$, 
there exists 
a positive constant $C^*$ independent of $f$ such that
\[
\|\mathcal{B}[f] \|_{H^{1}(\T^3)} \leq C^*\| f \|_{L^{2}(\T^3)}.
\]
Moreover, if $f \in L^2(\T^3)$ given by the form
$f = \nabla \cdot g$ for some $g \in \lt[ H^1(\T^3)\rt]^3$,
    then
    \[
    \| \mathcal{B}[f] \|_{L^{2}(\T^3)} \leq C^*\|g\|_{L^{2}(\T^3)}.
    \]
%
%
\end{lemma}
\begin{proof} 
The desired results immediately follow by the elliptic estimate \eqref{elliptic0}.
\end{proof}
We are now ready to present the proof of Proposition \ref{prop:large-t}.
\begin{proof}[Proof of Proposition \ref{prop:large-t}] We first define a modified temporal interacting energy-variation $\mathcal{E}^\sigma$ and the corresponding dissipation $\mathcal{D}^\sigma$ as follows.
\begin{align}\label{ep_s}
\begin{aligned}
\mathcal{E}^\sigma(t) &:= \mathcal{E}(t) - 2\sigma\int_{\T^3} (n+1)(v - j_c) \mb[n] dx,\cr
\mathcal{D}^\sigma(t) &:= \mathcal{D}(t) + \sigma\lt( \int_{\T^3} ((n+1)v\otimes v):\nabla \mb[n] + n((n+1)^\gamma - 1) - \mu \nabla v : \nabla \mb[n] - (\mu + \lambda)(\nabla \cdot v)n \,dx\rt)\cr
& \quad + \sigma\lt( \int_{\T^3} \rho (u-v)\mb[n] - (n+1)(v - j_c)\mb[\nabla \cdot ((n+1)v)] - ((n+1)j_c)_t \,\mb[n]\,dx\rt),
\end{aligned}
\end{align}
for any $\sigma > 0$. Then by a straightforward computation, we get 
\[
\frac12\frac{d}{dt}\mathcal{E}^\sigma(t) + \mathcal{D}^\sigma(t) = 0.
\]

For the rest of this section, for the sake of clarity, we distinguish the positive constants which will appear in the following estimates. 
We first shall show that $\mathcal{E}^\sigma$ is equivalent to $\mathcal{L}$, i.e., there are positive constants $c_1$ and $c_2$ such that 
\begin{equation}\label{equiv0}
c_1\mathcal{L}(t) \leq \mathcal{E}^\sigma(t) \leq c_2\mathcal{L}(t) \text{ for } \sigma > 0 \text{ sufficiently small}.
\end{equation}
More specifically, $c_1$ and $c_2$ are given by
\[
c_1 = \min \lt\{ 1 - \sigma \bar n, \frac{\rho_c}{\rho_c +1}, C_1(\bar n, \gamma) - \sigma C^* \rt\} \quad \mbox{and} \quad c_2 = \max \lt\{ 1 + \sigma \bar n, \frac{\rho_c}{\rho_c + 1}, C_2(\bar n, \gamma) + \sigma C^*\rt\}.
\]
To see this, we obtain the upper bound as
\begin{equation}
\begin{split}
\mathcal{E}^\sigma(t) &:= \mathcal{E}(t) - 2\sigma\int_{\T^3} (n+1)(v - j_c) \mb[n] dx \\
& \le \mathcal{E}(t) + \sigma\lt( \bar n\int_{\T^3} (n+1)|v-j_c|^2dx + C^* \int_{\T^3} n^2 dx\rt)\\
&=  \int_{\T^3} \rho |u - m_c|^2 dx + \int_{\T^3} (n+1) | v - j_c|^2 dx + \frac{2}{\gamma - 1}\int_{\T^3} (n+1)^\gamma dx +  \frac{\rho_c}{1 + \rho_c}|m_c - j_c|^2 \\
& + \sigma\lt( \bar n\int_{\T^3} (n+1)|v-j_c|^2dx + C^* \int_{\T^3} n^2 dx\rt)\\
&\le \max \lt\{ 1 + \sigma \bar n, \frac{\rho_c}{\rho_c + 1}, C_2(\bar n, \gamma) + \sigma C^*\rt\} \mathcal{L}(t).
\end{split}
\end{equation}
Here we've used Lemma \ref{oper-b-lem}:
\begin{equation}\label{sib}
\lt| \sigma\int_{\T^3} (n+1)(v - j_c)\mb[n] dx\rt| \leq \sigma\lt( \frac{\bar n}{2}\int_{\T^3} (n+1)|v-j_c|^2dx + \frac{C^*}{2}\int_{\T^3} n^2 dx\rt),
\end{equation}
the definition of $\mathcal{E}^\sigma(t)$ in \eqref{ep_s}, and
Lemma \ref{press-lem}-(1).
The lower bound can be proved similarly. 

We claim that there exists a positive constant $c_3 >0$ such that $\mathcal{L}(t) \leq c_3\mathcal{D}^\sigma(t)$ for sufficiently small $\sigma>0$. 
For the sake of clarity, we rewrite $\mathcal{D}^\sigma(t)$ as 
\begin{align*}
\begin{aligned}
\mathcal{D}^\sigma(t) &= \mu\int_{\T^3}|\nabla v|^2 dx + (\mu+ \lambda)\int_{\T^3} |\nabla \cdot v|^2 dx + \int_{\T^3} \rho |u-v|^2 dx+ \sigma\int_{\T^3} ((n+1)v\otimes v) : \nabla \mb[n]dx\cr
&+ \sigma \int_{\T^3}n((n+1)^\gamma - 1)dx- \sigma \mu \int_{\T^3} \nabla v : \nabla \mb[n]dx - \sigma(\mu + \lambda)\int_{\T^3} (\nabla \cdot v)n\,dx \cr
&+ \sigma\int_{\T^3}\rho (u-v)\cdot \mb[n]dx - \sigma\int_{\T^3}(n+1)(v-j_c)\cdot\mb[\nabla \cdot((n+1)v)]dx - \sigma \int_{\T^3}((n+1)j_c)_t\cdot\mb[n]dx\cr
& =: \sum_{i=1}^{10}I_i.
\end{aligned}
\end{align*}
For the estimate of $I_4$, we can rewrite it by adding and subtracting to get 
\begin{align*}
\begin{aligned}
I_4 &= \sigma \int_{\T^3} (n+1)(( v - j_c) \otimes v) : \nabla
\mb[n] dx + \sigma \int_{\T^3} (n+1)(j_c \otimes ( v -
j_c)) : \nabla \mb[n] dx \cr
&\quad + \sigma \int_{\T^3} n(j_c \otimes j_c) : \nabla \mb[n] dx \cr
 & =: I_{4}^1 + I_{4}^2 +  I_{4}^3.
\end{aligned}
\end{align*}
For the terms  $I_{4}^i$, $i=1,2,3$ can be estimated as 
\begin{align*}
\begin{aligned}
I_{4}^1 &\leq \frac{\sigma^{1/2}\bar n \|v\|_{L^\infty}}{2}
\int_{\T^3} (n+1)| v  - j_c|^2 dx + C^*\sigma^{3/2}\int_{\T^3} n^2 dx, \cr
I_{4}^2 &\leq \frac{\sigma^{1/2}\bar n E_0}{2}
\int_{\T^3} (n+1)| v  - j_c|^2 dx + C^*\sigma^{3/2}\int_{\T^3} n^2 dx, \cr
I_{4}^3 &\leq C^*\sigma E_0\int_{\T^3} n^2 dx,
\end{aligned}
\end{align*}
where we used Young's inequality together with \eqref{est-jc} and Lemma \ref{oper-b-lem}. Thus, we have
\[
I_4 \leq \lt(\frac{\sigma^{1/2}\bar n(\|v\|_{L^\infty} + E_0)}{2} \rt)\int_{\T^3} (n+1)|v-j_c|^2 dx + C^*\sigma(E_0 + \sigma^{1/2})\int_{\T^3}n^2 dx.
\]
For the estimate of $I_{5}$, by Taylor expansion, one can obtain
$$I_5  = \sigma\int_{\T^3}((n+1) - 1)((n+1)^\gamma - 1) dx \geq c_4\sigma\int_{\T^3} n^2 dx,$$
where $c_4$ depends only on $\bar n$. 

For the estimate of $I_{10}$, we notice that
\[
((n+1)j_c)_t = -\nabla \cdot((n+1)v) j_c + (n+1)j_c',
\]
and this yields that
\begin{align*}
\begin{aligned}
I_{10} &= -\sigma \int_{\T^3} j_c \cdot ((n+1)v \cdot \nabla \mb[n]) dx - \sigma \int_{\T^3} (n+1) j_c' \cdot \mb[n]dx\cr
& =: I_{10}^1 + I_{10}^2.
\end{aligned}
\end{align*}
Here $I_{10}^i,i=1,2$ are estimated by
\begin{align*}
\begin{aligned}
I_{10}^1 & = -\sigma \int_{\T^3} j_c\cdot ((n+1)(v - j_c) \cdot \nabla \mb[n]) dx -\sigma \int_{\T^3} j_c \cdot (nj_c \cdot \nabla \mb[n]) dx \cr
&\leq \frac{\sigma^{1/2}E_0\bar n}{2}\int_{\T^3} (n+1)|v-j_c|^2 dx + C^*\sigma^{3/2}\int_{\T^3} n^2 dx + C^*\sigma E_0\int_{\T^3} n^2 dx\cr
&\leq \frac{\sigma^{1/2}E_0\bar n}{2}\int_{\T^3} (n+1)|v-j_c|^2 dx + C^*\lt( \sigma^{3/2} + \sigma E_0\rt)\int_{\T^3} n^2 dx,\cr
I_{10}^2 &\leq \frac{\sigma^{1/2}\bar n^2}{2}|j_c'|^2 + C^*\sigma^{3/2}\int_{\T^3} n^2 dx\cr
&\leq \frac{\sigma^{1/2}\bar n^2 \rho_c}{2}\int_{\T^3} \rho |u-v|^2 dx + C^*\sigma^{3/2}\int_{\T^3} n^2 dx,
\end{aligned}
\end{align*}
where we used \eqref{est-jc} in Lemma \ref{jc-est-lem}. 

For $I_9$, we have that
\begin{align*}
\begin{aligned}
I_9 &= -\sigma\int_{\T^3} (n+1)(v-j_c)\mb[\nabla \cdot ((n+1)(v-j_c))]dx - \sigma\int_{\T^3} (n+1)(v-j_c)\mb[\nabla \cdot(nj_c)]dx\cr
&\leq C^*\sigma\bar n \int_{\T^3} (n+1)|v-j_c|^2 dx + \frac{\sigma^{1/2}\bar n}{2}\int_{\T^3} (n+1)|v-j_c|^2 dx + \frac{C^*\sigma^{3/2}|j_c|^2}{2}\int_{\T^3}n^2 dx\cr
&\leq \lt( C^*\sigma\bar n + \frac{\sigma^{1/2}\bar n}{2}\rt)\int_{\T^3} (n+1)|v-j_c|^2 dx + \frac{C^*\sigma^{3/2}E_0}{2}\int_{\T^3} n^2 dx,
\end{aligned}
\end{align*}
 Here we make use of Lemma \ref{oper-b-lem} for the second inequality.

Similarly as in the previous estimates, we can handle the rest of the terms as
\begin{align*}
\begin{aligned}
I_6 &\leq \frac\mu4\int_{\T^3} |\nabla v|^2 dx + C^*\sigma^2\mu \int_{\T^3} n^2 dx,\cr
I_7 &\leq \frac{\mu + \lambda}{4} \int_{\T^3} |\nabla \cdot v|^2 dx + C^*\sigma^2(\mu + \lambda)\int_{\T^3} n^2 dx,\cr
I_8 &\leq \frac14\int_{\T^3} \rho |u-v|^2 dx + C^*\sigma^2\bar\rho\int_{\T^3} n^2 dx.
\end{aligned}
\end{align*}
We now combine all the estimates above to find
\begin{align*}
\begin{aligned}
\mathcal{D}^\sigma(t) &\geq \frac\mu2\int_{\T^3} |\nabla v|^2 dx + \frac{\mu + \lambda}{2} \int_{\T^3} |\nabla \cdot v|^2 dx + \lt( \frac34 - \frac{\sigma^{1/2}\bar n^2 \rho_c}{2}\rt)\int_{\T^3} \rho |u-v|^2 dx\cr
& + c_5 \int_{\T^3} n^2 dx - c_6 \int_{\T^3} (n+1)|v - j_c|^2 dx,
\end{aligned}
\end{align*}
where $c_5,c_6$ are positive constants for $\sigma>0$ and $E_0 > 0$ small enough, and are given by
\begin{align*}
\begin{aligned}
c_5 &:= \sigma( c_4 - C^*(E_0 - \sigma^{1/2}(1 + E_0)- \sigma ( 2\mu + \lambda) + \sigma \bar \rho )), \quad c_6 := \sigma^{1/2}(\bar n (\|v\|_{L^\infty} + E_0) + C^*\sigma^{1/2}\bar n + \bar n).
\end{aligned}
\end{align*}
On the other hand, it follows from \eqref{lem:equiv1} that
\[
\frac12\int_{\T^3} (n+1)|v - j_c|^2 dx \leq 2\bar n c_p\int_{\T^3} |\nabla v|^2 dx,
\]
Then we obtain
\[
\mathcal{D}^\sigma(t) \geq c_7 \int_{\T^3} |\nabla v|^2 dx + c_8\int_{\T^3} \rho  |u-v|^2 dx + c_5 \int_{\T^3} n^2 dx,
\]
where 
\[
c_7 := \frac\mu2 - 4\bar n c_p c_6 \quad \mbox{and} \quad c_8 := \frac34 - \frac{\sigma^{1/2}\bar n^2 \rho_c}{2}.
\]
We again use the estimate in the proof of Lemma \ref{lem:dissi} to have
\begin{align*}
\begin{aligned}
\mathcal{D}^\sigma(t) &\geq \min\{c_7, c_8\}\lt( \int_{\T^3} |\nabla v|^2 dx + \int_{\T^3} \rho |u-v|^2 dx \rt) + c_5\int_{\T^3} n^2 dx\cr
&\geq  \frac{\min\{c_7,c_8\}}{2c_9}\lt( \int_{\T^3} \rho |u - m_c|^2 dx + \rho_c |m_c - j_c|^2 + \int_{\T^3} (n+1)|v - j_c|^2 dx\rt)  + c_5\int_{\T^3} n^2 dx\cr
&\geq \min\lt\{ \frac{\min\{c_7,c_8\}}{2c_9}c_{10} , c_5\rt\}\mathcal{L}(t),
\end{aligned}
\end{align*}
where
\[
c_9 := \max\lt\{1,  2c_p\lt(  3\lt( \rho_c \bar n + \bar \rho \rt) + \bar n\rt)\rt\} \quad \mbox{and} \quad c_{10} := \min\{\rho_c, 1\}.
\]
Hence, for $\sigma >0$ and $E_0 > 0$ small enough,  we have
\[
\frac{d}{dt}\mathcal{E}^\sigma(t) + \frac{c_{11}}{c_2}\mathcal{E}^\sigma(t) \leq 0,
\]
where $c_{11}$ is a positive constant given by
\[
c_{11} := \min\lt\{ \frac{\min\{c_7,c_8\}}{2c_9}c_{10} , c_5\rt\} > 0.
\]
This together with \eqref{equiv0} yields that
\[
\mathcal{L}(t) \leq \frac{\mathcal{E}^\sigma(t)}{c_1} \leq \frac{\mathcal{E}^\sigma(0)}{c_1}e^{-\frac{c_{11}}{c_2}t} \leq \frac{c_2}{c_1}\mathcal{L}(0) e^{-\frac{c_{11}}{c_2}t}, \quad t \geq 0.
\]
This completes the proof of Proposition \ref{prop:large-t}.
\end{proof}
%
%
%
\section{Global existence of the classical solutions} \label{sec5}
In this section, we provide the {\it a priori} estimates for global existence of the classical solutions to  \eqref{two-hydro-eqns-1}-\eqref{ini-two-hydro-eqns}. The  alignment estimate  between the two fluid velocities derived in Section \ref{sec4} plays a crucial role in obtaining the uniform bound for the density in the pressureless Euler equations. Using this we obtain the uniform estimates for the solution. This together with a standard continuation argument enable us to construct the global solution.
To this end, we first define
\[
\mathcal{S}(T;s) := \sup_{0 \leq t \leq T}\lt( \|u(t)\|_{H^{s}}^2 + \|n(t)\|_{H^{s}}^2 + \|v(t)\|_{H^{s}}^2\rt), 
\]
\[
\mathcal{S}^*(T;s) := \sup_{0 \leq t \leq T}\lt( \|\rho(t)\|_{H^s}^2 + \|u(t)\|_{H^{s+2}}^2 + \|n(t)\|_{H^{s+1}}^2 + \|v(t)\|_{H^{s+1}}^2\rt),
\]
and
\[
\mathcal{S}_0(s) := \|u_0\|_{H^s}^2 + \|n_0\|_{H^s}^2 + \|v_0\|_{H^s}^2, \quad \mathcal{S}^*_0(s) := \|\rho_0\|_{H^s}^2 + \|u_0\|_{H^{s+2}}^2 + \|n_0\|_{H^{s+1}}^2 + \|v_0\|_{H^{s+1}}^2.
\]
\newcommand{\eps}{\epsilon}
\begin{lemma}\label{lem:rn} Let $s> \frac52$ and $T > 0$ be given. Suppose $\mathcal{S}^*(T;s) \leq \epsilon_1$ for sufficiently small $\epsilon_1>0$.  Then we have
\bq\label{est:rho}
\sup_{0 \leq t \leq T}\|\rho(t)\|_{H^s} \leq C\|\rho_0\|_{H^s},
\eq
where $C$ is independent of $T$.
\end{lemma}
\begin{proof}We first notice that the estimate for the time-asymptotic behavior in Proposition \ref{prop:large-t} does not require any conditions on the lower bounds of $\rho$ and $n+1$. Here we obtain the uniform lower bound for $\rho$ as follows: Let $x(t),t\ge0$ be  a characteristic curve, that is, $x(t)$ is a solution to
$$\frac{d x(t)}{dt} = u(x(t), t), \ \ x(0)=x_0.$$
Then by the characteristic method and the smallness assumption $\|\nabla u\|_{H^{s+1}}\le \epsilon_1$, we have
\begin{equation*}
\begin{split}
\rho(x,t) &= \rho_0(x_0 ) 
\exp{\left(-\int_0^t \nabla\cdot u(x(s),s)) ds\right)} \ge \left( \min_{x\in\T^3} \rho_0(x) \right) \exp{\left( -t\| \nabla u \|_{L^\infty((0,T)\times \T^3)} \right) }
\ge \delta_0 e^{- \epsilon_1 t}.
\end{split}
\end{equation*}
This together with Proposition \ref{prop:large-t}, i.e., $\|\sqrt{\rho} (u-m_c) \|_{L^2} \le C e^{- \lambda t}$, we can deduce that
\[
\|u - m_c\|_{L^2} \le C\delta_0^{-1} e^{-(\lambda-\epsilon_1) t} \quad \mbox{for some} \quad \lambda > 0.
\]
Then we use a standard Sobolev inequality to get
\begin{equation}\label{nab_u}
\|\nabla u\|_{H^s} \leq \|u - m_c\|_{H^{s+1}} \leq C\|u - m_c\|_{L^2}^{1 - \beta}\|u-m_c\|_{H^{s+2}}^\beta \leq Ce^{-\tilde\lambda t},
\end{equation}
where $\beta := \frac{s+1}{s+2}\in(0,1)$ and $\tilde\lambda =( \lambda-\epsilon_1) (1-\beta)$ is a positive constant for sufficiently small $\eps_1>0$. This implies that $\nabla u \in L^1(0,T;H^s(\T^3))$ and, subsequently, we have
\[
\|\rho\|_{H^s}^2 \leq \|\rho_0\|_{H^s}^2\exp\lt(C\int_0^t \|\nabla u(\tau)\|_{H^s} d\tau\rt)\leq \|\rho_0\|_{H^s}^2\exp\lt(C\int_0^t e^{-\tilde\lambda \tau} d\tau\rt) \leq C\|\rho_0\|_{H^s}^2.
\]
Here we used the Gronwall inequality to estimate  $\|\rho\|^2_{H^s}$ as in the proof of Lemma \ref{lem:ext-linear1}.
\end{proof}

\begin{corollary}\label{low-bd-rho} Under the same assumptions as in Lemma \ref{lem:rn}, there exists a constant $\underline \rho > 0$ such that
\[
\rho(x,t) \geq \underline \rho \quad \mbox{for all} \quad (x,t) \in \T^3 \times [0,T],
\]
where $\underline \rho$ is independent of $T$.
\end{corollary}
\begin{proof} 
Revisiting the characteristic method as in Lemma \ref{lem:rn} together with a sharper estimate \eqref{nab_u}, we have
\begin{equation*}
\begin{split}
\rho(x,t) &= \rho_0(x(0) ) 
\exp{\left(-\int_0^t \nabla\cdot u(x(s),s)) ds\right)} \\
&\ge \left( \min_{x\in\T^3} \rho_0(x) \right) \exp{\left( -t\| \nabla u \|_{L^\infty((0,T)\times \T^3)} \right) }\\
&
\ge \delta_0 \exp{( - \epsilon_0 t \exp({-\tilde\lambda t})) } \ge \delta_0 \min_{t\ge0}  \left( \exp{( - \epsilon_0 t \exp({-\tilde\lambda t})) } \right) \ge \underline\rho \ \ \text{ for all } t\ge0.
\end{split}
\end{equation*}
\end{proof}
\begin{remark}\label{rmk:rho} 
1. In order to obtain the uniform boundedness of $\rho$ in $H^s(\T^3)$-norm, we need to show that $u \in H^{s+1+\epsilon_2}(\T^3)$ for any $\epsilon_2 > 0$.

2. One can also find from Lemma \ref{lem:rn} that
\[
\sup_{0 \leq t \leq T}\|\rho(t)\|_{H^k} \leq C\|\rho_0\|_{H^k} \quad \mbox{for} \quad k \in [0,s].
\] 
\end{remark}
We next show the uniform boundedness of the rest of terms in $\mathcal{S}^*(T;s)$ in ascending order with respect to the space-derivative of the solutions. We define
\[
E^0(n,v) := \frac{1}{\gamma-1}\lt( (1+n)^\gamma - 1 - \gamma n\rt) + \frac12(1+n)|v|^2.
\]
\begin{lemma}\label{lem:e} 
For $\gamma > 0$ and $\|n\|_{L^\infty} \leq \epsilon_1 \leq \frac12$,  there exists a positive constant $C>0$ such that
\[
C^{-1}\lt( n^2 + |v|^2\rt) \leq E^0(n,v) \leq C\lt( n^2 + |v|^2\rt),
\]
i.e., it is simply denoted as $
E^0(n,v) \approx n^2 + |v|^2$.
\end{lemma}
\begin{proof}
It follows from Taylor expansion that
\[
\frac{1}{\gamma-1}\lt( (1+n)^\gamma - 1 - \gamma n\rt) = \frac\gamma2 n^2 + \frac{\gamma(\gamma-2)}{6}n^3(1+\xi n)^{\gamma - 3} \quad \mbox{for some } \xi \in (0,1).
\]
This yields that for $\|n\|_{L^\infty} \leq \epsilon_1 \leq \frac12$
\[
\frac{1}{\gamma-1}\lt( (1+n)^\gamma - 1 - \gamma n\rt) \geq \frac{n^2}{2}\lt( \gamma - \frac{|\gamma(\gamma-2)|}{3}2^{|\gamma-3|}n\rt) \geq C_{\epsilon_1,\gamma} \frac{n^2}{2}.
\]
The  upper bound of $E^0(n,v)$ can be obtained similarly. This yields the result.
\end{proof}
Then we now provide two lemmas on the estimates of the zeroth- and first-order derivative of the solutions. Since we already have that $\|\rho(t)\|_{H^s}\le C \| \rho_0\|_{H^s}$ in Lemma \ref{lem:rn}, we shall focus only on the estimates for the unknown functions, $u,n$, and $v$.
\begin{lemma}\label{lem:sl2} 
Let $s>\frac52$ and $T>0$ be given. Suppose that $\mathcal{S}^*(T;s) \leq \epsilon_1$ for sufficiently small $\eps_1>0$. Then we have
\[
\mathcal{S}(T;0) \leq C\mathcal{S}_0(0),
\]
where $C$ is independent of $T$.
\end{lemma}
\begin{proof}
Note that
\begin{align*}
\begin{aligned}
\frac{d}{dt}\int_{\T^3}E^0(n,v) dx&= \frac{1}{\gamma-1}\int_{\T^3}\lt( \gamma(1+n)^{\gamma-1}n_t + \gamma n_t\rt)dx + \frac12\int_{\T^3}n_t|v|^2 dx + \int_{\T^3}(1+n)v\cdot v_t \,dx\cr
&=: I_1 + I_2 + I_3,
\end{aligned}
\end{align*}
where $I_i,i=1,2,3$ are estimated as follows.
\begin{align*}
\begin{aligned}
I_1 & = \frac{1}{\gamma-1}\int_{\T^3} \lt( \gamma(1+n)^{\gamma-1} + \gamma\rt)(-\nabla n \cdot v - (1+n)\nabla \cdot v)dx =  \int_{\T^3} \nabla p(1+n) \cdot v \,dx,\cr
I_2 & = \int_{\T^3} (n+1)v \cdot \lt(v \cdot \nabla v\rt) dx, \cr
I_3 & = - \int_{\T^3} (n+1)v \cdot \lt(v \cdot \nabla v\rt) + \nabla p(1+n)\cdot v dx - \mu \int_{\T^3} |\nabla v|^2 dx - (\mu + \lambda)\int_{\T^3} |\nabla \cdot v|^2 dx \cr
&\quad + \int_{\T^3} \rho (u-v)\cdot v \,dx.
\end{aligned}
\end{align*}
Combining the estimates, we obtain
\begin{equation*}
\frac{d}{dt}\int_{\T^3} E^0(n,v) dx + \mu \int_{\T^3} |\nabla v|^2 dx + (\mu + \lambda)\int_{\T^3} |\nabla \cdot v|^2 dx = \int_{\T^3} \rho (u-v)\cdot v \,dx.
\end{equation*}
Furthermore, using the momentum equations for $u$, we have 
$$\frac{d}{dt} \int_{\T^3} \frac12\rho|u|^2 dx = -\int_{\T^3} \rho(u-v)\cdot u dx.$$
Thus, we have
\begin{align*}
\begin{aligned}
\frac{d}{dt}\lt( \int_{\T^3} E^0(n,v) + \frac12 \rho|u|^2 dx\rt) + \mu \int_{\T^3}|\nabla v|^2 dx + (\mu + \lambda)\int_{\T^3} |\nabla \cdot v|^2 dx + \int_{\T^3} \rho |u-v|^2 dx = 0,
\end{aligned}
\end{align*}
and this yields, upon integration over $(0,t)$, that
\begin{align*}
\begin{aligned}
&\int_{\T^3} E^0(n,v) dx + \int_{\T^3} \frac12 \rho|u|^2 dx + \mu \int_0^t\int_{\T^3}|\nabla v|^2 dx ds + (\mu + \lambda)\int_0^t\int_{\T^3} |\nabla \cdot v|^2 dx ds + \int_0^t \int_{\T^3} \rho |u-v|^2 dxds \cr
&\qquad \leq \int_{\T^3} E^0(n_0,v_0) dx + \frac12\int_{\T^3} \rho_0|u_0|^2 dx.
\end{aligned}
\end{align*}
Hence, by Lemma \ref{lem:e},  this implies that 
\[
\sup_{0 \leq t \leq T} \lt(\|n(t)\|_{L^2}^2 + \|v(t)\|_{L^2}^2 + \int_0^t\|\nabla v(s)\|_{L^2}^2ds \rt) \ls \|n_0\|_{L^2}^2 + \|v_0\|_{L^2}^2 + \|u_0\|_{L^2}^2.
\]
For the estimate of $\|u\|_{L^\infty(0,T;L^2)}$, one can easily obtain
\begin{align*}
\begin{aligned}
\frac12\frac{d}{dt}\|u\|_{L^2}^2 &\leq -\lt( \frac12 - C\epsilon_1\rt)\|u\|_{L^2}^2 + \frac12\|v\|_{L^2}^2 \leq -\lt( \frac12 - C\epsilon_1\rt)\|u\|_{L^2}^2 + C\lt(\|n_0\|_{L^2}^2 + \|v_0\|_{L^2}^2 + \|u_0\|_{L^2}^2\rt),
\end{aligned}
\end{align*}
this yields, by Gronwall's inequality, that
\[
\|u\|_{L^2}^2 \le C \lt(\|n_0\|_{L^2}^2 + \|v_0\|_{L^2}^2 + \|u_0\|_{L^2}^2\rt).
\]
Thus we arrive at
\[
\sup_{0 \leq t \leq T} \lt(\|u(t)\|_{L^2}^2 + \|n(t)\|_{L^2}^2 + \|v(t)\|_{L^2}^2 + \int_0^t\|\nabla v(s)\|_{L^2}^2ds \rt) \ls \|n_0\|_{L^2}^2 + \|v_0\|_{L^2}^2 + \|u_0\|_{L^2}^2.
\]
Combining this with  Remark \ref{rmk:rho}.2, we have
\[
\mathcal{S}(T;0) \leq C\mathcal{S}_0(0).
\]
\end{proof}
\begin{lemma}\label{est-uvnh1}
Let $s>\frac52$ and $T>0$ be given. Suppose $\mathcal{S}^*(T;s) \leq \epsilon_1 \ll 1$. Then we have
\[
\mathcal{S}(T;1) \leq C\mathcal{S}_0(1).
\]
\end{lemma}
\begin{proof}
First we easily find that
\begin{align*}
\begin{aligned}
\sup_{0 \leq t \leq T}\lt( \|\nabla u(t)\|_{L^2}^2 + C(1 - C\epsilon_1) \int_0^t \|\nabla u(s)\|_{L^2}^2 ds \rt)&\leq C\|\nabla u_0\|_{L^2}^2 + C\int_0^t\|\nabla v(s)\|_{L^2}^2 ds \cr
&\leq C\|\nabla u_0\|_{L^2}^2 + C\mathcal{S}_0(0)\cr
&\leq C\mathcal{S}_0(1).
\end{aligned}
\end{align*}
For the estimate of $\|\nabla n\|_{L^2}$, we consider
\[
\frac{d}{dt}\int_{\T^3} \nabla n \cdot \lt( \frac{\nabla n}{2} + \frac{(n+1)^2}{2\mu + \lambda}v \rt) dx.
\]
Here, for notational simplicity, we omit the summation, i.e., $f_i \,g_i := \sum_{i=1}^3 f_i \,g_i$. Then we get
\begin{align*}
\begin{aligned}
&\frac{d}{dt}\int_{\T^3} \pa_i n \lt( \frac{\pa_i n}{2} + \frac{(n+1)^2}{2\mu + \lambda} v_i\rt) dx \cr
& \,\, = \int_{\T^3} \pa_i n_t \lt( \pa_i n + \frac{(n+1)^2}{2\mu + \lambda}v_i\rt) dx + \int_{\T^3} \pa_i n \lt( \frac{(n+1)^2}{2\mu + \lambda} \pa_t v_i\rt) dx + \int_{\T^3} \pa_i n \lt( \frac{2(n+1)}{2\mu + \lambda}v_i n_t\rt) dx\cr
& \,\, =: J_1 + J_2 + J_3.
\end{aligned}
\end{align*}
First, using the momentum equations for $v$, we estimate $J_2$ as
\begin{align*}
\begin{aligned}
J_2 &= -\frac{1}{2\mu + \lambda}\int_{\T^3} (n+1)^2 (\pa_i n) v_j \pa_j v_i \,dx - \frac{\gamma}{2\mu +\lambda} \int_{\T^3} (n+1)^\gamma |\pa_i n|^2 dx  + \frac{\mu}{2\mu + \lambda}\int_{\T^3} (n+1)\pa_i n \pa_j^2v_i \,dx \cr
&\quad + \frac{\mu + \lambda}{2\mu + \lambda}\int_{\T^3} (n+1)\pa_i n \pa_{ij} v_j \,dx + \frac{1}{2\mu + \lambda}\int_{\T^3} (n+1)(\pa_i n) \rho(u_i - v_i) dx.
\end{aligned}
\end{align*}
Note that
\begin{align*}
\begin{aligned}
&\frac{\mu}{2\mu + \lambda}\int_{\T^3} (n+1)\pa_i n \pa_j^2v_i \,dx + \frac{\mu + \lambda}{2\mu + \lambda}\int_{\T^3} (n+1)\pa_i n \, \pa_{ij} v_j \,dx\cr
& \quad = \frac{1}{2\mu + \lambda}\int_{\T^3} -\mu (\pa_j n \,\pa_i n + (n+1)\pa_{ij} n)\pa_j v_i - (\mu + \lambda)(\pa_j n \,\pa_i n + (n+1)\pa_{ij}n)\pa_i v_j \,dx\cr
&\quad = - \int_{\T^3} \pa_j n\, \pa_i n \,\pa_j v_i \,dx - \int_{\T^3} (n+1)\pa_{ij}n \,\pa_j v_i \,dx \cr
&\quad = \int_{\T^3} (n+1)\pa_j n \,\pa_{ij}v_i \,dx.
\end{aligned}
\end{align*}
Using this, we have
\begin{align*}
\begin{aligned}
J_2 &= -\frac{1}{2\mu + \lambda}\int_{\T^3} (n+1)^2 (\pa_i n) v_j \pa_j v_i \,dx - \frac{\gamma}{2\mu +\lambda} \int_{\T^3} (n+1)^\gamma |\pa_i n|^2 dx\cr
&\quad + \int_{\T^3} (n+1)\pa_j n \pa_{ij}v_i \,dx +  \frac{1}{2\mu + \lambda}\int_{\T^3} (n+1)(\pa_i n) \rho(u_i - v_i) dx.
\end{aligned}
\end{align*}
Next we split $J_1$ into two parts:
\[
J_1^1 := \int_{\T^3} \pa_i n_t \pa_i n \,dx \quad \mbox{and} \quad J_1^2 := \int_{\T^3} \pa_i n_t \frac{(n+1)^2}{2\mu + \lambda} v_i \,dx.
\]
Then we find
\begin{align*}
\begin{aligned}
J_1^1 &= -\frac12\int_{\T^3} |\pa_i n|^2 \pa_j v_j \,dx - \int_{\T^3} \pa_i n \pa_j n \pa_i v_j \,dx - \int_{\T^3} (n+1)\pa_i n \pa_{ij}v_j \,dx,
\end{aligned}
\end{align*}
and
\begin{align*}
\begin{aligned}
J_1^2 &= - \frac{1}{2\mu + \lambda}\int_{\T^3} (n+1)^2 v_i \lt( \pa_{ij}n v_j + \pa_j n \pa_i v_j + \pa_i n \pa_j v_j + (n+1)\pa_{ij}v_j \rt) dx\cr
&= \frac{1}{2\mu + \lambda} \int_{\T^3} \lt( 2(n+1)\pa_i n v_i v_j + (n+1)^2 \pa_i v_i v_j + (n+1)^2 v_i \pa_i v_j \rt) \pa_j n \,dx\cr
&\quad - \frac{1}{2\mu + \lambda} \int_{\T^3} (n+1)^2 v_i \pa_j n \pa_i v_j \,dx - \frac{1}{2\mu + \lambda} \int_{\T^3} (n+1)^2 v_i \pa_i n \pa_j v_j \,dx\cr
&\quad + \frac{1}{2\mu + \lambda} \int_{\T^3} \lt( 3(n+1)^2 \pa_i n v_i + (n+1)^3 \pa_i v_i \rt) \pa_j v_j \,dx\cr
&= \frac{1}{2\mu + \lambda} \int_{\T^3} 2(n+1)\pa_i n \pa_j n  v_i v_j \,dx + \frac{1}{2\mu + \lambda} \int_{\T^3} \lt( 3(n+1)^2 \pa_i n v_i + (n+1)^3 \pa_i v_i \rt) \pa_j v_j \,dx.
\end{aligned}
\end{align*}
Thus we obtain
\begin{align*}
\begin{aligned}
J_1^1 + J_2 &= -\frac12\int_{\T^3} |\pa_i n|^2 \pa_j v_j \,dx - \int_{\T^3} \pa_i n \pa_j n \pa_i v_j \,dx - \frac{1}{2\mu + \lambda} \int_{\T^3}(n+1)^2 \pa_i n v_j \pa_j v_i \,dx\cr
&\quad - \frac{\gamma}{2\mu + \lambda}\int_{\T^3}(n+1)^\gamma |\pa_i n|^2 \,dx + \frac{1}{2\mu + \lambda} \int_{\T^3} (n+1)\pa_i n \rho (u_i - v_i)dx,\cr
J_1^2 + J_3 &= \frac{1}{2\mu + \lambda} \int_{\T^3} (n+1)^2 \pa_i n v_i \pa_j v_j \,dx + \frac{1}{2\mu + \lambda}\int_{\T^3} (n+1)^3 \pa_i v_i \pa_j v_j \,dx.
\end{aligned}
\end{align*}
Now we combine all estimates above to have
\begin{align*}
\begin{aligned}
&\frac{d}{dt}\int_{\T^3} \nabla n \cdot \lt( \frac{\nabla n}{2} + \frac{(n+1)^2 }{2\mu + \lambda} v\rt) dx  \cr
&\quad =  - \frac12 \int_{\T^3} |\nabla n|^2 \nabla \cdot v \,dx - \int_{\T^3} \nabla n \cdot (\nabla v \cdot \nabla n) dx - \frac{1}{2\mu + \lambda} \int_{\T^3}(n+1)^2 \nabla n \cdot (v \cdot \nabla v) dx \cr
&\qquad - \frac{\gamma}{2\mu + \lambda}\int_{\T^3} (n+1)^\gamma|\nabla n|^2 dx + \frac{1}{2\mu + \lambda}\int_{\T^3} (n+1)\nabla n \cdot \rho (u-v) dx \cr
&\qquad + \frac{1}{2\mu + \lambda}\int_{\T^3} (n+1)^2 (\nabla n \cdot v) (\nabla \cdot v)dx + \frac{1}{2\mu + \lambda}\int_{\T^3} (n+1)^3 |\nabla \cdot v|^2 dx\cr
&\leq C\epsilon_1 \|\nabla n\|_{L^2}^2 + C\epsilon_1\|\nabla n\|_{L^2}\|\nabla v\|_{L^2} - \frac{\gamma}{2\mu + \lambda}\|\nabla n\|_{L^2}^2 + C\sqrt{\epsilon_1}\|\nabla n\|_{L^2}\|\sqrt{\rho}|u - v|\|_{L^2}+ C\|\nabla v\|_{L^2}^2\cr
&\leq -C(1 - C\sqrt{\epsilon_1})\|\nabla n\|_{L^2}^2 + C\sqrt{\epsilon_1}\|\sqrt{\rho}|u - v|\|_{L^2}^2 + C\|\nabla v\|_{L^2}^2.
\end{aligned}
\end{align*}
Here we have used $S^*(T;s)\le\eps_1$. By integrating it over $(0,t)$ with respect to $t$, we obtain
\begin{align*}
\begin{aligned}
&\frac12\int_{\T^3} |\nabla n|^2 dx + \frac{1}{2\mu + \lambda}\int_{\T^3} (n+1)^2 \nabla n \cdot v \,dx + C(1 - C\sqrt{\epsilon_1})\int_0^t\|\nabla n\|_{L^2}^2 ds  \cr
& \quad \leq \frac12\|\nabla n_0\|_{L^2}^2 + \frac{1}{2\mu + \lambda}\|(n_0+1)\sqrt{|\nabla n_0||v_0|}\|_{L^2}^2 + C\sqrt{\epsilon_1}\int_0^t \|\sqrt{\rho}|u-v|\|_{L^2}^2 ds+ C\int_0^t \|\nabla v\|_{L^2}^2 ds\cr
&\quad \leq C\|\nabla u_0\|_{L^2}^2 + C\mathcal{S}_0(0)\leq C\mathcal{S}_0(1),
\end{aligned}
\end{align*}
where the $L^2$ estimate has been used. We also notice that
\[
\lt|\frac{1}{2\mu + \lambda}\int_{\T^3} (n+1)^2 \nabla n \cdot v \,dx \rt|\leq \frac14\int_{\T^3} |\nabla n|^2 dx + C\int_{\T^3} |v|^2 dx \leq \frac14\int_{\T^3} |\nabla n|^2 dx + C\mathcal{S}_0(0),
\]
and this implies
\[
\sup_{0 \leq t \leq T}\lt( \|\nabla n(t)\|_{L^2}^2 + C_0\int_0^t\|\nabla n(s)\|_{L^2}^2 ds \rt)\leq C\mathcal{S}_0(1).
\]
So far, we have shown that
\[
\sup_{0 \leq t \leq T} \lt( \|\nabla u(t)\|_{L^2}^2 + \|\nabla n(t)\|_{L^2}^2 + C_0\lt( \int_0^t \|\nabla u(s)\|_{L^2}^2 + \|\nabla n(s)\|_{L^2}^2 ds \rt)\rt) \leq C\mathcal{S}_0(1),
\]
where $C_0$ and $C$ are positive constants independent of $t$. 

We finally estimate $v$ in the $H^1$-norm. Note that $v$ satisfies 
\[
v_t + v \cdot \nabla v + \frac{\nabla p(n+1)}{n+1} + \frac{Lv}{n+1} = \frac{\rho}{n+1}(u-v).
\]
A straightforward computation yields that
\begin{align*}
\begin{aligned}
\frac12\frac{d}{dt}\int_{\T^3} |\nabla v|^2 dx &= -\int_{\T^3} \nabla v \cdot \lt(\nabla (v \cdot \nabla v) + \nabla \lt( \frac{\nabla p}{n+1}\rt) + \nabla \lt(\frac{Lv}{n+1} \rt) \rt) dx+  \int_{\T^3} \nabla v \cdot \nabla \lt(\frac{\rho}{n+1}(u-v)\rt) dx\cr
&=:\sum_{j=1}^4 K_j.
\end{aligned}
\end{align*}
Here $K_j,j=1,\dots,4$ are estimated as follows.
\begin{align*}
\begin{aligned}
K_1 &= -\int_{\T^3} \nabla v \cdot \lt( \nabla v \cdot \nabla v + v \cdot \nabla^2 v\rt) dx \le C \|\nabla v\|_{L^\infty} \|\nabla v\|_{L^2}^2 \leq C\epsilon_1\|\nabla v\|_{L^2}^2,
\end{aligned}
\end{align*}
\begin{align*}
\begin{aligned}
K_2 &= \int_{\T^3} \nabla^2 v \cdot \lt( \frac{\nabla p}{n+1}\rt) dx \leq \frac\mu4 \int_{\T^3} \frac{|\nabla^2 v|^2}{n+1}dx + \frac1\mu\int_{\T^3} \frac{|\nabla p|^2}{n+1}dx \le \frac\mu4 \int_{\T^3} \frac{|\nabla^2 v|^2}{n+1}dx + \frac{C\gamma}{\mu}\|\nabla n\|_{L^2}^2,
\end{aligned}
\end{align*}
\begin{align*}
\begin{aligned}
K_3 &= \int_{\T^3} \nabla v \cdot \nabla \lt( \frac{\mu}{n+1} \Delta v + \frac{\mu + \lambda}{n+1}\nabla \nabla \cdot v \rt)dx \leq -\mu\int_{\T^3} \frac{|\nabla^2 v|^2}{n+1} dx + C\epsilon_1\|\nabla^2 v\|_{L^2}^2 + C\epsilon_1\|\nabla v\|_{L^2}^2,
\end{aligned}
\end{align*}
and
\begin{align*}
\begin{aligned}
K_4 &= \int_{\T^3} \Delta v  \cdot \lt(\frac{\rho}{n+1}(u-v)\rt) dx \leq C\sqrt{\epsilon_1}\|\nabla^2 v\|_{L^2}\|\sqrt{\rho}|u-v|\|_{L^2}\leq C\epsilon_1\|\nabla^2 v\|_{L^2}^2 + C\|\sqrt{\rho}|u-v|\|_{L^2}^2.\cr
\end{aligned}
\end{align*}
Thus, combining all the estimates, we obtain
\bq\label{est-vh1}
\frac{d}{dt}\int_{\T^3} |\nabla v|^2 dx + C(\mu - C\sqrt{\epsilon_1})\int_{\T^3} |\nabla^2 v|^2dx \leq  \frac{C\gamma}{\mu}\|\nabla n\|_{L^2}^2 + C\|\sqrt{\rho}(u-v)\|_{L^2}^2 + C\epsilon_1\|\nabla v\|_{L^2}^2,
\eq
and by integrating it over $(0,t)$  with respect to $t$, we arrive at
\begin{align*}
\begin{aligned}
\|\nabla v\|_{L^2}^2 + C(\mu - C\epsilon_1)\int_0^t\|\nabla^2 v\|_{L^2}^2 ds
&\leq \|\nabla v_0\|_{L^2}^2 + C\int_0^t \|\nabla v\|_{L^2}^2 + \|\nabla n\|_{L^2}^2 + \|\sqrt{\rho}|u-v|\|_{L^2}^2 ds \leq C\mathcal{S}_0(1).
\end{aligned}
\end{align*}
Thus we have
\[
\sup_{0 \leq t \leq T}\lt( \|\nabla v(t)\|_{L^2}^2 + \int_0^t \|\nabla^2 v(s)\|_{L^2}^2 ds \rt) \leq C\mathcal{S}_0(1),
\]
where $C$ is independent of $t$. Hence by combining Lemmas \ref{lem:rn} and \ref{lem:sl2} we conclude that
\[
\mathcal{S}(T;1) \leq C\mathcal{S}_0(1).
\]
\end{proof}
\begin{remark}(Refined estimates for $v$) Notice that
\begin{align*}
\begin{aligned}
\frac12\frac{d}{dt}\|v\|_{L^2}^2 &= -\int_{\T^3} v \cdot \lt( v \cdot \nabla v + \frac{\nabla p(n+1)}{n+1} + \frac{Lv}{n+1}\rt) dx + \int_{\T^3} \frac{\rho}{n+1}(u-v) \cdot v dx\cr
&=: \sum_{i=1}^4 I_i,
\end{aligned}
\end{align*}
where $I_i,i=1,\cdots, 4$ are estimated as follows.
\begin{align*}
\begin{aligned}
I_1 &\leq C\epsilon_1 \|v\|_{L^2}^2 \leq C\mathcal{S}_0(0),\cr
I_2 &\leq C\|v\|_{L^2}\|\nabla p\|_{L^2} \leq C\|v\|_{L^2}\|\nabla n\|_{L^2} \leq C\mathcal{S}_0(1),
\end{aligned}
\end{align*}
\begin{align*}
\begin{aligned}
I_3 &= -\int_{\T^3} \mu \nabla \lt( \frac{v}{n+1}\rt) \cdot \nabla v - (\mu + \lambda)\nabla \cdot \lt( \frac{v}{n+1}\rt) \nabla \cdot v \,dx\cr
&= -\int_{\T^3} \frac{1}{n+1}\lt(\mu |\nabla v|^2 + (\mu + \lambda)|\nabla \cdot v|^2\rt) dx+ \int_{\T^3} \frac{1}{(n+1)^2}\lt( \mu v \nabla n \cdot \nabla v + (\mu + \lambda) v \cdot \nabla n \nabla \cdot v \rt)dx\cr
&\leq -C\mu\int_{\T^3} |\nabla v|^2 dx + C\epsilon_1\|\nabla n\|_{L^2}^2 \leq C\mathcal{S}_0(1),
\end{aligned}
\end{align*}
and 
\begin{align*}
\begin{aligned}
I_4 &\leq \frac12\int_{\T^3} \frac{\rho}{n+1}|u|^2 dx - \frac12\int_{\T^3} \frac{\rho}{n+1}|v|^2 dx \leq  C\mathcal{S}_0(0).
\end{aligned}
\end{align*}
Thus we find that
\[
\frac{d}{dt}\|v\|_{L^2}^2 + \|v\|_{L^2}^2 + C\mu\|\nabla v\|_{L^2}^2 \leq C\mathcal{S}_0(1).
\]
We can also deduce from \eqref{est-vh1} that
\bq\label{est-refine-vh1}
\frac{d}{dt}\|\nabla v\|_{L^2}^2 + C(\mu - C\epsilon_1)\|\nabla^2 v\|_{L^2}^2 \leq C\mathcal{S}_0(1).
\eq
\end{remark}
We now provide the high-order estimates of $u,v$ and $n$. The proofs of these are lengthy and technical, thus we postpone them to Appendix \ref{app:high}.
\begin{lemma}\label{lem:u} Let $s>\frac52$. Suppose $S^*(T;s) \leq \epsilon_1 \ll 1$. Then, for any $2 \leq k \leq s+2$, we obtain
\[
\frac{d}{dt}\|\nabla^k u\|_{L^2}^2 + C_1\|\nabla^k u\|_{L^2}^2 \leq  \|\nabla^k v\|_{L^2}^2,
\]
where $C_1 : = C(1 - \epsilon_1)$ is a positive constant independent of $t$.
\end{lemma}
\begin{lemma}\label{lem:v} Let the same assumptions hold as in Lemma \ref{lem:u}. Then, for any $2 \leq k \leq s+1$, we obtain
\begin{align*}
\begin{aligned}
&\frac{d}{dt}\|\nabla^{k} v\|_{L^2}^2 +  C_2\|\nabla^{k+1} v\|_{L^2}^2 \leq  C\|\nabla n\|_{H^{k-1}}^2 + C\epsilon_1\lt( \|\nabla^2 u\|_{L^2}^2 + \|\nabla^{k-1} u\|_{L^2}^2\rt) +C\mathcal{S}_0(1),
\end{aligned}
\end{align*}
where $C_2 := C(\mu - C\epsilon_1)$ is a positive constant independent of $t$.
\end{lemma}
\begin{lemma}\label{lem:n} Let $T$ be given. Suppose $S^*(T;s) \leq \epsilon_1 \ll 1$. Then for any $2 \leq k \leq s+1$ we obtain
\begin{align*}
\begin{aligned}
&\frac{d}{dt}\int_{\T^3} \nabla^k n \cdot \lt( \frac{\nabla^k n}{2} + \frac{(n+1)^2}{2\mu + \lambda} \nabla^{k-1} v\rt) dx + C_3\|\nabla^k n\|_{L^2}^2\cr
&\qquad \leq C\epsilon_1\|\nabla^2 u\|_{H^{k-2}}^2 + C\epsilon_1\|\nabla v\|_{H^k}^2 + C\|\nabla^k v\|_{L^2}^2 + C\mathcal{S}_0(1),
\end{aligned}
\end{align*}
where $C_3:= \frac{C\gamma}{2\mu + \lambda}$ is a positive constant independent of $t$.
\end{lemma}
\begin{proposition}\label{prop:appri}
Let $T$ be given. Suppose $\mathcal{S}^*(T;s) \leq \epsilon_1 \ll 1$. Then we have
\[
\mathcal{S}^*(T;s) \leq C_0 \mathcal{S}^*_0(s).
\]
\end{proposition}
\begin{proof}
We set 
\[
E_1(t) :=\|\nabla v(t)\|_{L^2}^2,
\]
and 
\[
E_k(t) := \delta_k^{-1}\|\nabla^k v\|_{L^2}^2 + \|\nabla^k u\|_{L^2}^2 + \int_{\T^3} \nabla^k n \cdot \lt( \frac{\nabla^k n}{2} + \frac{(n+1)^2}{2\mu + \lambda} \nabla^{k-1} v\rt) dx,
\]
for $2\leq k \leq s+1$, where $\delta_k > 0$ will be determined later. We now claim that for $2 \leq k \leq s+1$, there exist positive constants $\delta_k > 0$ such that 
\bq\label{eq_equ1}
E_k(t) + \sum_{1\leq n \leq k-1}E_n \prod_{n+1 \leq r \leq k}\delta_r \approx \|\nabla v\|_{H^{k-1}}^2 + \|\nabla^2 n\|_{H^{k-2}}^2 + \|\nabla^2 u\|_{H^{k-2}}^2,
\eq
and
\begin{align}\label{eq_equ2}
\begin{aligned}
&\frac{d}{dt}\lt( E_k(t) + \sum_{1\leq n \leq l-1}E_n \prod_{n+1 \leq r \leq k}\delta_r\rt)\cr
&\qquad \leq -C(\delta_2,\cdots,\delta_k, \epsilon_1)\sum_{2\leq l \leq k}\lt( \|\nabla^l v\|_{H^1}^2 + \|\nabla^l u\|_{L^2}^2 + \|\nabla^l n\|_{L^2}^2 \rt)  + C\mathcal{S}_0(1).
\end{aligned}
\end{align}
For the proof of the claim, we use the inductive argument on $k$. We first show that \eqref{eq_equ1} and \eqref{eq_equ2} hold for $k=2$. Since 
\[
\lt|\int_{\T^3} \nabla^2 n \cdot  \frac{(n+1)^2}{2\mu + \lambda} \nabla v\, dx\rt| \leq \frac14\|\nabla^2 n\|_{L^2}^2 + C\|\nabla v\|_{L^2}^2,
\]
we have 
\[
E_2 \geq \delta_2^{-1}\|\nabla^2 v\|_{L^2}^2 + \|\nabla^2 u\|_{L^2}^2 + \frac14 \|\nabla^2 n\|_{L^2}^2 - C\|\nabla v\|_{L^2}^2.
\]
Thus by choosing a sufficiently large positive constant $\delta_2 > 0$ satisfying $\delta_2 > C$ we get
\[
E_2 + \delta_2 E_1 \geq \delta_2^{-1}\|\nabla^2 v\|_{L^2}^2 + \|\nabla^2 u\|_{L^2}^2 + \frac14 \|\nabla^2 n\|_{L^2}^2 + (\delta_2 - C)\|\nabla v\|_{L^2}^2.
\]
For the upper bound of $E_2 + \delta_2E_1$, one can easily obtain
\[
E_2 + \delta_2 E_1 \leq \delta_2^{-1}\|\nabla^2 v\|_{L^2}^2 + \|\nabla^2 u\|_{L^2}^2 + \frac34 \|\nabla^2 n\|_{L^2}^2 + C\|\nabla v\|_{L^2}^2.
\]
Thus \eqref{eq_equ1} holds for $k=2$. Next, by combining  \eqref{est-refine-vh1} and the estimates in Lemmas \ref{lem:u} - \ref{lem:n}, we estimate 
\begin{align*}
\begin{aligned}
\frac{d}{dt}\lt(E_2(t) + \delta_2E_1(t)\rt)&\leq -\lt(  C_2\delta_2 -C \rt)\|\nabla^2 v\|_{L^2}^2 - \lt( C_2 - C\epsilon_1\rt)\|\nabla^3 v\|_{L^2}^2 - \lt(C_1 - C\epsilon_1\delta_2^{-1} - C\epsilon_1 \rt)\|\nabla^2 u\|_{L^2}^2\cr
&\quad - \lt(C_3 - C\delta_2^{-1}\rt)\|\nabla^2 n\|_{L^2}^2 + C\mathcal{S}_0(1).
\end{aligned}
\end{align*}
We then choose a positive constant $\delta_2 > 0$ large enough to have
\[
\frac{d}{dt}\lt(E_2(t) + \delta_2E_1(t) \rt) \leq -C(\delta_2, \epsilon_1)\lt(\|\nabla^2 v\|_{H^1}^2 + \|\nabla^2 u\|_{L^2}^2 + \|\nabla^2 n\|_{L^2}^2\rt) + C\mathcal{S}_0(1),
\]
where $C(\delta_2, \epsilon_1) >0$ is a constant depending only on $\delta_2$ and $ \epsilon_1$. 
Now we assume that \eqref{eq_equ1} and \eqref{eq_equ2} hold for $k=m \leq s$. Then we notice that
 \[
 E_{m+1}(t) + \sum_{1\leq n \leq m}E_n \prod_{n+1 \leq r \leq m+1}\delta_r =  E_{m+1}(t) + \delta_{m+1}\lt( E_m + \sum_{1\leq n \leq m-1}E_n\prod_{n+1\leq r \leq m}\delta_r\rt),
 \]
and where 
\[
E_m + \sum_{1\leq n \leq m-1}E_n\prod_{n+1\leq r \leq m}\delta_r \approx \|\nabla v\|_{H^{m-1}}^2 + \|\nabla^2 n\|_{H^{m-2}}^2 + \|\nabla^2 u\|_{H^{m-2}}^2.
\]
On the other hand, $E_{m+1}$ is estimated by
\[
E_{m+1} \geq \delta_{m+1}^{-1}\|\nabla^{m+1} v\|_{L^2}^2 + \|\nabla^{m+1} u\|_{L^2}^2 + \frac14 \|\nabla^{m+1} n\|_{L^2}^2 - C\|\nabla^m v\|_{L^2}^2.
\]
Thus, we choose a positive constant $\delta_{m+1} > 0$ large enough such that $\delta_{m+1} - C >0$, and this gives the desired lower bound for $E_{m+1}$. Similarly, we can easily find the upper bound of $E_{m+1}$. Hence, we have that \eqref{eq_equ1} holds for $2 \leq k \leq s+1$. Next we estimate 
\begin{align}\label{est_eq0}
\begin{aligned}
\frac{d}{dt}\lt(E_{m+1}(t) + \sum_{1\leq n \leq m}E_n \prod_{n+1 \leq r \leq m+1}\delta_r \rt) &= \frac{d}{dt}E_{m+1}(t) + \delta_{m+1}\frac{d}{dt}\lt( E_m + \sum_{1\leq n \leq m-1}E_n \prod_{n+1 \leq r \leq m}\delta_r\rt) \cr
& =:I_1 + I_2.
\end{aligned}
\end{align}
Here, for the estimate of $I_1$, we again use the estimates in Lemmas \ref{lem:u} - \ref{lem:n} to get
\begin{align}\label{est_eq1}
\begin{aligned}
I_1 &\leq -\lt(C_3 - C\delta_{m+1}^{-1} \rt)\|\nabla^{m+1} n\|_{L^2}^2 - \lt(C_1 - C\epsilon_1 \rt)\|\nabla^{m+1} u\|_{L^2}^2 \cr
&\quad - \lt(C_2\delta_{m+1}^{-1} - C\epsilon_1 \rt)\|\nabla^{m+2} v\|_{L^2}^2 + C\|\nabla^{m+1} v\|_{L^2}^2 + C\mathcal{S}_0(1).
\end{aligned}
\end{align}
On the other hand, it follows from our assumption that there exist positive constants $\delta_2,\cdots,\delta_m > 0$ such that 
\begin{align}\label{est_eq2}
\begin{aligned}
I_2 &\leq -\delta_{m+1}C(\delta_2, \cdots, \delta_m,\epsilon_1)\sum_{2\leq l \leq m}\lt(\|\nabla^l v\|_{H^1}^2 + \|\nabla^l u\|_{L^2}^2 + \|\nabla^l n\|_{L^2}^2\rt)  + C\mathcal{S}_0(1),
\end{aligned}
\end{align}
where $C(\delta_2, \cdots, \delta_m,\epsilon_1)$ is a positive constant independent of $t$. Then we combine the estimates \eqref{est_eq0}-\eqref{est_eq2} to have
\begin{align*}
\begin{aligned}
&\frac{d}{dt}\lt(E_{m+1}(t) + \sum_{1\leq n \leq m}E_n \prod_{n+1 \leq r \leq m+1}\delta_r \rt)\cr
&\quad \leq -\lt(C_3 - C\delta_{m+1}^{-1} \rt)\|\nabla^{m+1} n\|_{L^2}^2 - \lt(C_1 - C\epsilon_1 \rt)\|\nabla^{m+1} u\|_{L^2}^2 - \lt(C_2\delta_{m+1}^{-1} - C\epsilon_1 \rt)\|\nabla^{m+2} v\|_{L^2}^2 \cr
&\qquad -\lt(\delta_{m+1}C(\delta_2, \cdots, \delta_m) - C\rt)\|\nabla^{m+1} v\|_{L^2}^2\cr
&\qquad -\delta_{m+1}C(\delta_2, \cdots, \delta_m)\sum_{2\leq l \leq m}\lt(\|\nabla^l v\|_{L^2}^2 + \|\nabla^l u\|_{L^2}^2 + \|\nabla^l n\|_{L^2}^2\rt) + C\mathcal{S}_0(1),
\end{aligned}
\end{align*}
and this yields that, for $\delta_{m+1} > 0$ large enough, there exists a  constant $C(\delta_2,\cdots,\delta_{m+1},\epsilon_1) > 0$ such that
\begin{align*}
\begin{aligned}
&\frac{d}{dt}\lt( E_{m+1}(t) + \sum_{1\leq n \leq m}E_n \prod_{n+1 \leq r \leq m+1}\delta_r\rt)\cr
&\qquad \leq -C(\delta_2,\cdots,\delta_{m+1}, \epsilon_1)\sum_{2\leq l \leq m+1}\lt( \|\nabla^l v\|_{H^1}^2 + \|\nabla^l u\|_{L^2}^2 + \|\nabla^l n\|_{L^2}^2 \rt)  + C\mathcal{S}_0(1).
\end{aligned}
\end{align*}
This completes the proof of the claim. Next we set
\[
\mathcal{X}_{s+1}(t) := E_{s+1}(t) + \sum_{1\leq n \leq s}E_n \prod_{n+1 \leq r \leq s+1}\delta_r,
\]
and
\[
\mathcal{Y}_{s+1}(t) := \|\nabla v\|_{H^{s}}^2 + \|\nabla^2 u\|_{H^{s-1}}^2 + \|\nabla^2 n\|_{H^{s-1}}^2.
\]
Then it follows from \eqref{eq_equ1} and \eqref{eq_equ2} that

\[
\frac{d}{dt}\mathcal{X}_{s+1}(t) + C \mathcal{X}_{s+1}(t) \leq C\mathcal{S}_0(s+1),
\]
and this again implies that
\bq\label{est_y}
\mathcal{Y}_{s+1}(t) \leq C\mathcal{X}_{s+1}(t) \leq C\mathcal{S}_0(s+1),
\eq
where we used $\mathcal{X}_{s+1}(0) \leq C\mathcal{S}_0(s+1)$. For the estimate of $\|\nabla^{s+2} u\|_{L^2}$, it follows from Lemmas \ref{lem:u} and \ref{lem:v} that
\begin{align*}
\begin{aligned}
\frac{d}{dt}\lt(\|\nabla^{s+1} v\|_{L^2}^2 + C_2\|\nabla^{s+2} u\|_{L^2}^2\rt) 
&\leq -C_1C_2\|\nabla^{s+2} u\|_{L^2}^2 + C\|\nabla n\|_{H^s}^2 + C\epsilon_1\|\nabla^s u\|_{L^2}^2 + C\mathcal{S}_0(1)\cr
&\leq -C_1C_2\|\nabla^{s+2} u\|_{L^2}^2 - C\|\nabla^{s+1} v\|_{L^2}^2 + C\mathcal{S}_0(s+1),
\end{aligned}
\end{align*}
where we used the estimate \eqref{est_y}. Finally, we apply the Gronwall's inequality and combine the estimates \eqref{est:rho} and in Lemma \ref{est-uvnh1}  to conclude 
\[
\mathcal{S}^*(T;s) \leq C\mathcal{S}^*_0(s).
\]
This completes the proof.
\end{proof}
\begin{proof}[Proof of Theorem \ref{thm:main}] We choose a positive constant 
$
M := \min\{ \epsilon_0, \epsilon_1\},
$
where $\epsilon_0$ and $\epsilon_1$ are positive constants appeared in Theorem \ref{thm:local} and Proposition \ref{prop:appri}, respectively. We also choose the initial data $W_0 = (\rho_0,u_0,n_0,v_0)$ satisfying
\[
\|W_0\|_{\mn^s} \leq \frac{M}{2\sqrt{1 + C_0}},
\]
where $C_0 > 0$ is given in Proposition \ref{prop:appri}. We now define the lifespan of solutions to the system \eqref{two-hydro-eqns-1}-\eqref{ini-two-hydro-eqns}:
\[
T := \sup\lt\{t \ge 0\ \ : \sup_{0 \leq \tau \leq t}\|W(\tau)\|_{\mn^s} < M \rt\}.
\]
Since 
\[
\|W_0\|_{\mn^s} \leq \frac{M}{2\sqrt{1 + C_0}} \leq \frac M2 \leq \epsilon_0,
\]
 Theorem \ref{thm:local} implies that $T > 0$. Suppose that $T < +\infty$. Then we can deduce from the definition of $T$ and Theorem \ref{thm:local} that
\bq\label{contra}
\sup_{0 \leq \tau \leq T}\|W(\tau)\|_{\mn^s} = M.
\eq
On the other hand, it follows from Proposition \ref{prop:appri} that
\[
\sup_{0 \leq \tau \leq T}\|W(\tau)\|_{\mn^s} \leq \sqrt{C_0}\|W_0\|_{\mn^s} \leq \frac{M\sqrt{C_0}}{2\sqrt{1 + C_0}} \leq \frac M2,
\]
which is a contradiction to \eqref{contra}. Hence, one can conclude that $T = \infty$. 
This completes the proof.
\end{proof}
%
%
%
%
\appendix
%
%
%
%
%
%
%
%
\section{Some proofs for local existence}\label{Local}
\begin{proof}[Proof of Lemma \ref{prop:invar} ]
$\bullet$ Step 1($\rho$-estimate): We first show the positivity of $\rho$. For this, we define the characteristic $x(s) := x(s;t,x)$ which solves the following differential equations:
\[
\frac{dx(s)}{ds} = \bar u(s,x(s)),
\]
with $x(t) = x$. Then we can easily find that for $0 \leq t \leq T_0$
\[
\rho(x,t) = \rho_0(x_0)\exp\lt( -\int_0^t \nabla \cdot \bar u(s,x(s))ds\rt) \geq \rho_0(x_0)\exp\lt( -CM_0T_0\rt) > 0,
\]
where we used $\|\nabla \bar u(s,x(s))\|_{L^\infty} \leq C\sup_{0 \leq t \leq T_0}\|\nabla \bar u\|_{H^2} \leq CM_0$.

We next show the uniform boundedness of $\rho$. For any $0\leq k \leq s$, it follows from $\eqref{li-sys-a}_1$ that
\[
\nabla^k \pa_t\rho = -\nabla^k(\bar u \cdot \nabla \rho) - \nabla^k(\rho\nabla\cdot \bar u).
\]
This yields that
\begin{align}\label{es:li-rho-1}
\begin{aligned}
\frac12 \frac{d}{dt}\|\nabla^k \rho\|_{L^2}^2 &= -\frac12\int_{\T^3} \bar u \cdot \nabla|\nabla^k \rho|^2 dx - \int_{\T^3} [\nabla^k, \bar u \cdot \nabla]\rho \cdot \nabla^k \rho \,dx\cr
&\quad - \int_{\T^3} \rho \nabla^k (\nabla \cdot \bar u)\cdot \nabla^k \rho \,dx - \int_{\T^3}[\nabla^k, \rho\nabla\cdot]\bar u \cdot \nabla^k \rho \,dx\cr
&\quad := \sum_{i=1}^4 I_i,
\end{aligned}
\end{align}
where $[A,B]$ denotes the commutator operator, i.e., $[A,B] := AB - BA$.
Here $I_i,i=1,\dots,4$ are estimated as follows.
\begin{align}\label{es:li-rho-2}
\begin{aligned}
I_1 &\leq \frac12\int_{\T^3}(\nabla \cdot \bar u)|\nabla^k \rho|^2 dx \leq \|\nabla \bar u\|_{L^\infty}\|\nabla^k \rho\|_{L^2}^2 \leq CM_0\|\nabla^k \rho\|_{L^2}^2,\cr
I_2 &\leq \|[\nabla^k, \bar u \cdot \nabla]\rho\|_{L^2}\|\nabla^k \rho\|_{L^2} \leq CM\|\rho\|_{H^s}\|\nabla^k \rho\|_{L^2},\cr
I_3 &\leq \|\rho\|_{L^\infty}\|\nabla^{k+1}\bar u\|_{L^2}\|\nabla^k \rho\|_{L^2} \leq CM_0\|\rho\|_{H^s}\|\nabla^k \rho\|_{L^2},\cr
I_4 &\leq \|[\nabla^k, \rho\nabla \cdot]\bar u\|_{L^2}\|\nabla^k \rho\|_{L^2} \leq  CM_0\|\rho\|_{H^s}\|\nabla^k \rho\|_{L^2}.
\end{aligned}
\end{align}
We now combine \eqref{es:li-rho-1} and \eqref{es:li-rho-2}, and sum over $k$ to find
\[
\frac{d}{dt} \|\rho\|_{H^s}^2 \leq CM_0\|\rho\|_{H^s}^2,
\]
and this deduces that
\[
\sup_{0 \leq t \leq T_0}\|\rho(t)\|_{H^s} \leq \|\rho_0\|_{H^s}e^{CM_0T_0}.
\]
We then choose properly small positive constants $\epsilon_0, T_0, M_0$ such that 
\[
\sup_{0 \leq t \leq T_0}\|\rho(t)\|_{H^s} \leq M_0.
\]
$\bullet$ Step 2($u$-estimate): We first notice from the positivity of $\rho$ obtained in Step 1 that $u$ satisfies
\[
\pa_t u + \bar u \cdot \nabla u = -(\bar u - \bar v).
\]
Then for $0 \leq k \leq s+2$ we have
\[
\nabla^k \pa_t u = -\nabla^k(\bar u \cdot \nabla u) - \nabla^k(\bar u - \bar v) = -\bar u \cdot \nabla^{k+1} u -[\nabla^k, \bar u \cdot \nabla]u - \nabla^k(\bar u - \bar v),
\]
and this deduces
\begin{align*}
\begin{aligned}
\frac12 \frac{d}{dt}\|\nabla^k u\|_{L^2} &\leq \frac12\int_{\T^3}(\nabla \cdot \bar u)|\nabla^k u|^2 dx + \|[\nabla^k, \bar u \cdot \nabla]u\|_{L^2}\|\nabla^k u\|_{L^2} + \|\nabla^k \bar u\|_{L^2}\|\nabla^k u\|_{L^2} + \|\nabla^k \bar v\|_{L^2}\|\nabla^k u\|_{L^2}\cr
&\leq CM_0\|\nabla^k u\|_{L^2}^2 + CM_0\|\nabla u\|_{H^{s-1}}\|\nabla^k u\|_{L^2} + CM_0\|\nabla^k u\|_{L^2}+ \|\nabla^k \bar v\|_{L^2}\|\nabla^k u\|_{L^2}\cr
&\leq C(1 + M_0)\|\nabla^k u\|_{L^2}^2 + CM_0 + C\|\nabla^k \bar v\|_{L^2}^2.
\end{aligned}
\end{align*}
Then we now sum over $k$ to find
\[
\frac{d}{dt}\|u\|_{H^{s+2}}^2 \leq C(1 + M_0)\|u\|_{H^{s+2}}^2 + CM_0 + C\|\bar v\|_{H^{s+2}}^2,
\]
and this yields
\[
\|u\|_{H^{s+2}}^2 \leq \|u_0\|_{H^{s+2}}^2e^{C(1 + M_0)T_0} + e^{C(1 + M_0)T_0}\lt( CM_0T_0 + C\int_0^{T_0}\|\bar v\|_{H^{s+2}}^2 ds \rt).
\]
Since $\bar v \in L^2(0,T_0;H^{s+2})$, we can again choose small positive constants $\epsilon_0, T_0, M_0$ such that 
\[
\sup_{0 \leq t \leq T_0}\|u\|_{H^{s+2}} \leq M_0.
\]
$\bullet$ Step 3($w$-estimate): Similar fashion with the estimate for positivity of $\rho$ in Step 1, we can also find the positivity of $\bar n + 1$. Furthermore we obtain that there exists a positive constant $c_0 > 0$ such that
\[
c_0^{-1}\mathbb{I}_{4 \times 4} \leq A^0(\bar w) \leq c_0 \mathbb{I}_{4 \times 4} \quad \mbox{for} \quad t \in [0,T_0].
\]
We notice that $A^0(\bar w) \in L^\infty(0,T_0;H^{s+1}(\T^3))$ and the constant $c_0$ is independent of $M_0$ and $T_0$. We first estimate $L^2$-norm of $w$. It follows from $\eqref{li-sys-a}_3$ that
\begin{align*}
\begin{aligned}
&\frac12\lt(\pa_t \lag A^0(\bar w)w,w \rag  - \lag \pa_t A^0(\bar w)w,w \rag \rt) + \frac12\sum_{j=1}^3 \lt( \pa_j \lag A^j(\bar w)w,w\rag - \lag \pa_j A^j(\bar w)w,w \rag \rt)\cr
&\qquad = \lag A^0(\bar w)E_1(\bar n, v), w\rag + \lag A^0(\bar w)E_2(\bar \rho, \bar u, \bar n, \bar v),w\rag.
\end{aligned}
\end{align*}
Here the two terms in the right hand side of the above equation are estimated as follows.
\begin{align*}
\begin{aligned}
&\lag A^0(\bar w)E_1(\bar n, v), w\rag = -\int_{\T^3} v \cdot Lv\, dx = -\mu\int_{\T^3} |\nabla v|^2 dx - (\mu + \lambda) \int_{\T^3} |\nabla \cdot v|^2 dx,\cr
&\lag A^0(\bar w)E_2(\bar \rho, \bar u, \bar n, \bar v),w\rag = \int_{\T^3} v \cdot \bar \rho (\bar u - \bar v) \,dx \leq \|\bar \rho\|_{L^\infty}\|v\|_{L^2}\lt(\|\bar u\|_{L^2} + \|\bar v \|_{L^2}\rt) \leq CM_0\|v\|_{L^2}.
\end{aligned}
\end{align*}
This yields 
\begin{align*}
\begin{aligned}
\|w(t)\|_{L^2}^2 &\leq c_0\int_{\T^3} \lag A^0(\bar w) w, w\rag \big|_{t=t} \,dx\cr
&\leq c_0\int_{\T^3} \lag A^0(\bar w) w, w\rag \big|_{t=0} \,dx + CM_0\int_0^t \|w(s)\|_{L^2}^2 ds + CM_0\int_0^t \|v(s)\|_{L^2} ds\cr
& \quad - \mu\int_0^t \|\nabla v(s)\|_{L^2}^2 ds - (\mu + \lambda)\int_0^t \|\nabla \cdot v(s)\|_{L^2}^2 ds,
\end{aligned}
\end{align*}
and
\[
\|w(t)\|_{L^2}^2 \leq c_0^2\|w_0\|_{L^2}^2 + CM_0T_0+ CM_0\int_0^t \|w(s)\|_{L^2}^2 ds.
\]
Then we obtain
\[
\sup_{0\leq t \leq T_0}\|w(t)\|_{L^2}^2 \leq \lt( c_0^2 \|w_0\|_{L^2}^2 + CM_0T_0\rt)e^{CM_0T_0}.
\]
Thus we have
\[
\sup_{0\leq t \leq T_0}\|w(t)\|_{L^2} \leq M_0.
\]
by choosing small positive constants $\epsilon_0, T_0, M_0$ suitably.
Similarly, for $ 1\leq k \leq s+1$, we find
\[
A^0(\bar w)\pa_t \nabla^k w + \sum_{j=1}^3 A^j(\bar w)\pa_j \nabla^k w = R_k(\bar w, w) + A^0(\bar w) \nabla^k E_1(\bar n, v) + A^0(\bar w)\nabla^k E_2(\overline W),
\]
where
\[
R_k(\bar w, w) = -A^0(\bar w)\lt( [\nabla^k, A^0(\bar w)^{-1}\sum_{j=1}^3 A^j(\bar w)\pa_j]w\rt).
\]
Thus we obtain
\begin{align}\label{all}
\begin{aligned}
&\frac12 \lt( \pa_t \lag A^0(\bar w)\nabla^k w, \nabla^k w \rag - \lag \pa_t A^0(\bar w)\nabla^k w, \nabla^k w \rag\rt) + \frac12\sum_{j=1}^3 \lt( \pa_j \lag A^j(\bar w)\nabla^k w, \nabla^k w\rag - \lag \pa_j A^j(\bar w)\nabla^k w, \nabla^k w\rag\rt)\cr
&\qquad = \lag R_k(\bar w,w), \nabla^k w \rag + \lag A^0(\bar w)\nabla^kE_1(\bar n, v), \nabla^k w\rag + \lag A^0(\bar w)\nabla^k E_2(\bar W), \nabla^k w\rag.
\end{aligned}
\end{align}
Here we first estimate 
\begin{align}\label{part}
\begin{aligned}
\lag \pa_t A^0(\bar w)\nabla^k w, \nabla^k w \rag &\leq \|\pa_t A^0(\bar w)\|_{L^\infty}\|\nabla^k w\|_{L^2}^2 \leq CM_0\|\nabla^k w\|_{L^2}^2,\cr
\sum_{j=1}^3 \lag \pa_j A^j(\bar w)\nabla^k w, \nabla^k w\rag &\leq \sum_{j=1}^3\|\pa_j A^j(\bar w)\|_{L^\infty}\|\nabla^k w\|_{L^2}^2 \leq CM_0\|\nabla^k w\|_{L^2}^2,\cr
\lag R_k(\bar w,w), \nabla^k w \rag &\leq \|R_k(\bar w, w)\|_{L^2}\|\nabla^k w\|_{L^2}\leq CM_0\|\nabla^k w\|_{L^2}^2,
\end{aligned}
\end{align}
where we used
\begin{align*}
\begin{aligned}
\|R_k(\bar w, w)\|_{L^2} &\leq \|A^0(\bar w)\|_{L^\infty}\sum_{j=1}^3\|[\nabla^k, A^0(\bar w)^{-1}\pa_j]w\|_{L^2}\cr
&\leq C\sum_{j=1}^3\lt(\|\nabla(A^0(\bar w)^{-1}A^j(\bar w))\|_{L^\infty}\|\nabla^{k-1}\pa_j w\|_{L^2} + \|\nabla^k(A^0(\bar w)^{-1}A^j(\bar w))\|_{L^2}\|\pa_j w\|_{L^\infty}\rt)\cr
&\leq CM_0\|w\|_{H^{s+1}}.
\end{aligned}
\end{align*}
For the rest of terms, we find 
\begin{align}\label{rest-1}
\begin{aligned}
&\lag A^0(\bar w)\nabla^kE_1(\bar n, v), \nabla^k w\rag\cr
&\quad = \mu \int_{\T^3} (1 + \bar n)\nabla^k\lt( \frac{\nabla \nabla \cdot v}{1 + \bar n} \rt) \cdot \nabla^k v \,dx + (\mu + \lambda) \int_{\T^3} (1 + \bar n)\nabla^k \lt( \frac{\nabla \cdot \nabla v}{1 + \bar n}\rt) \cdot \nabla^k v\,dx\cr
&\quad =:J_1 + J_2.
\end{aligned}
\end{align}
We again decompose $J_1$ into three terms 
\begin{align*}
\begin{aligned}
J_1 &= \mu\int_{\T^3} \nabla \cdot \nabla^{k+1}v \cdot \nabla^k v\, dx + \mu \int_{\T^3}(1 + \bar n)\nabla^k\lt(\frac{1}{1 + \bar n}\rt)\nabla\cdot\nabla v \cdot \nabla^k v\,dx\cr
&\quad + \mu\sum_{1 \leq l \leq k-1}\binom{k}{l}\int_{\T^3}(1+\bar n)\nabla^l \lt(\frac{1}{1+\bar n}\rt)\nabla^{k-l}(\nabla \cdot \nabla v) \cdot \nabla^k v\,dx\cr
&=: J_1^1 + J_1^2 + J_1^3,
\end{aligned}
\end{align*}
and $J_1^i,i=1,2,3$ are estimated by
\begin{align*}
\begin{aligned}
J_1^1 &= -\mu\int_{\T^3} |\nabla^{k+1} v |^2 dx = -\mu\|\nabla^{k+1}v\|_{L^2}^2,\cr
J_1^2 &\leq \mu\|1 + \bar n\|_{L^\infty}\lt\|\nabla^k\lt(\frac{1}{1 + \bar n} \rt)\rt\|_{L^2}\|\nabla^2 v\|_{L^4}\|\nabla^k v\|_{L^4} \cr
&\leq C\|\bar n\|_{H^k}\|\nabla^2 v\|_{H^1}\|\nabla^k v\|_{H^1} \leq CM_0\|\nabla^2 v\|_{H^1}\|\nabla^k v\|_{H^1},\cr
J_1^3 &\leq C\|1 + \bar n\|_{L^\infty}\sum_{1 \leq l \leq k-1}\|\nabla^k v\|_{L^4}\lt\| \nabla^l \lt(\frac{1}{1 + \bar n}\rt)\rt\|_{L^4}\|\nabla^{k+2-l}v\|_{L^2}\cr
&\leq C\|\nabla^k v\|_{H^1}\|\bar n\|_{H^k} \|\nabla^2 v\|_{H^{k-1}} \leq CM_0\|\nabla^k v\|_{H^1}\|\nabla^2 v\|_{H^{k-1}}.
\end{aligned}
\end{align*}
This implies
\bq\label{rest-1-1}
J_1 \leq -\mu\|\nabla^{k+1}v\|_{L^2}^2 + CM_0(\|\nabla^2 v\|_{H^1} + \|\nabla^2 v\|_{H^{k-1}})\|\nabla^k v\|_{H^1}.
\eq
Similarly, we can also obtain
\bq\label{rest-1-2}
J_2 \leq -(\mu + \lambda)\|\nabla^k \nabla \cdot v\|_{L^2}^2 + CM_0(\|\nabla^2 v\|_{H^1} + \|\nabla^2 v\|_{H^{k-1}})\|\nabla^k v\|_{H^1}.
\eq
By combining \eqref{rest-1}, \eqref{rest-1-1}, and \eqref{rest-1-2}, we have
\begin{align}\label{rest1}
\begin{aligned}
\lag A^0(\bar w)\nabla^kE_1(\bar n, v), \nabla^k w\rag &\leq -\mu\|\nabla^{k+1}v\|_{L^2}^2 -(\mu + \lambda)\|\nabla^k \nabla \cdot v\|_{L^2}^2 + CM_0(\|\nabla^2 v\|_{H^1} + \|\nabla^2 v\|_{H^{k-1}})\|\nabla^k v\|_{H^1}.
\end{aligned}
\end{align}
We next estimate
\begin{align}\label{rest-2-1}
\begin{aligned}
\lag A^0(\bar w)\nabla^k E_2(\overline W), \nabla^k w\rag &= \int_{\T^3} (1 + \bar n)\nabla^k \lt( \frac{1}{1 + \bar n}\bar \rho (\bar u - \bar v)\rt) \cdot \nabla^k v\,dx\cr
&= \int_{\T^3} \nabla^{k-1}\lt( \frac{1}{1 + \bar n}\bar \rho(\bar u - \bar v)\rt) \cdot \nabla\lt( (1 + \bar n) \cdot \nabla^k v\rt)dx\cr
&= \int_{\T^3} \nabla^{k-1}\lt( \frac{1}{1 + \bar n}\bar \rho(\bar u - \bar v)\rt) \cdot \lt( \nabla \bar n \cdot \nabla^k v + (1 + \bar n)\nabla^{k+1} v\rt)dx\cr
&\leq \lt\|\nabla^{k-1}\lt( \frac{1}{1 + \bar n}\bar \rho(\bar u - \bar v)\rt)\rt\|_{L^2}\lt( \|\nabla \bar n\|_{L^\infty}\|\nabla^k v\|_{L^2} + \|1 + \bar n\|_{L^\infty}\|\nabla^{k+1}v\|_{L^2}\rt)\cr
&\leq CM_0\lt(M_0 + \|v\|_{H^2} + \|v\|_{H^{k-1}}\rt)\lt( \|\nabla^k v\|_{L^2} + \|\nabla^{k+1}v\|_{L^2}\rt),
\end{aligned}
\end{align}
where we used
\begin{align*}
\begin{aligned}
\lt\|\nabla^{k-1}\lt( \frac{1}{1 + \bar n}\bar \rho(\bar u - \bar v)\rt)\rt\|_{L^2} &\leq \lt\|\frac{\bar \rho}{1 + \bar n}\rt\|_{H^{k-1}}\|u-v\|_{L^\infty} + \lt\| \frac{\bar \rho}{1 + \bar n}\rt\|_{L^\infty}\|u-v\|_{H^{k-1}}\cr
&\leq \|\bar\rho\|_{H^{k-1}}\|1/(1 + \bar n)\|_{L^\infty}\|u-v\|_{H^2} + \|\bar \rho\|_{L^\infty}\|1/(1+\bar n)\|_{H^{k-1}}\|u-v\|_{H^{k-1}}\cr
&\leq CM_0\lt(M_0 + \|v\|_{H^2} + \|v\|_{H^{k-1}}\rt).
\end{aligned}
\end{align*}
Now we collect the estimates \eqref{all}, \eqref{part}, \eqref{rest1}, and \eqref{rest-2-1} to get
\begin{align*}
\begin{aligned}
\|\nabla^k w(t)\|_{L^2}^2 &\leq c_0\int_{\T^3} \lag A^0(\bar w)\nabla^k w, \nabla^k w\rag \big|_{t=t} \,dx\cr
&\leq c_0\int_{\T^3}\lag A^0(\bar w)\nabla^k w, \nabla^k w\rag \big|_{t=0} \,dx + CM_0\int_0^t \|\nabla^k w(s)\|_{L^2}^2 ds\cr
& \,\,\, + CM_0\int_0^t \lt(\|\nabla^2 v(s)\|_{H^1} + \|\nabla^2 v(s)\|_{H^{k-1}}\rt)\|\nabla^k v(s)\|_{H^1} ds\cr
&\,\,\, + CM_0\int_0^t\lt(M_0 + \|v(s)\|_{H^2} + \|v(s)\|_{H^{k-1}}\rt)\lt( \|\nabla^k v(s)\|_{L^2} + \|\nabla^{k+1}v(s)\|_{L^2}\rt)ds\cr
&\,\,\, - c_0\mu \int_0^t \|\nabla^{k+1}v(s)\|_{L^2}^2 ds - c_0(\mu + \lambda)\int_0^t \|\nabla^k \nabla \cdot v(s)\|_{L^2}^2 ds,
\end{aligned}
\end{align*}
and by summing over $k$ we have
\begin{align*}
\begin{aligned}
\|\nabla w(t)\|_{H^s}^2 &\leq C\|\nabla w_0\|_{H^s}^2 + CM_0^3T_0 + CM_0\int_0^t \|\nabla v(s)\|_{H^s}^2 ds- \lt( c_0 \mu - CM_0\rt)\int_0^t \|\nabla^2 v(s)\|_{H^s}^2 ds,
\end{aligned}
\end{align*}
due to $\|v\|_{L^2} \leq M_0$. Since $\|\nabla v\|_{H^s} \leq \|\nabla w\|_{H^s}$, we get
\[
\sup_{0 \leq t \leq T_0}\|\nabla w(t)\|_{H^s}^2 \leq\lt( C\|\nabla w_0\|_{H^s}^2 + CM_0^3T_0\rt)e^{CM_0T_0}.
\]
for $M_0$ small enough. Hence this concludes 
\[
\sup_{0 \leq t \leq T_0}\|\nabla w(t)\|_{H^s} \leq M_0,
\]
by selecting small positive constants $\epsilon_0, T_0, M_0$.
\end{proof}
\begin{proof}[Proof of Lemma \ref{lem:cauchy}] 


For notational simplicity, we set
$w^m = (n^m,v^m)$, 
\[
A^0_m := A^0(w^m), \quad A^j_m:= A^j(w^m), \quad E_1^m:=E_1(n^m,v^{m+1}), \quad E_2^m:= E_2(W^m),
\]
\[
\rho^{m+1,m}:= \rho^{m+1} - \rho^m, \quad u^{m+1,m}:= u^{m+1} - u^m, \quad \mbox{and} \quad w^{m+1,m}:=w^{m+1} - w^m.
\]
Then it follows from $\eqref{approx-sys}_1$ that
\[
\pa_t \rmmm + \umm \cdot \nabla \rho^{m+1} + u^{m-1}\cdot \nabla \rmmm + \rmmm \nabla \cdot u^m + \rho^m \nabla \cdot \umm = 0.
\]
Upon $L^2$ estimate, we get
\begin{align*}
\begin{aligned}
\frac12\frac{d}{dt}\|\rmmm\|_{L^2}^2 &= - \int_{\T^3} (\umm) \cdot \rho^{m+1} \rmmm \,dx + \frac12 \int_{\T^3} (\nabla \cdot u^{m-1})|\rmmm|^2 dx\cr
&\quad - \int_{\T^3} (\nabla \cdot u^m)|\rmmm|^2 dx - \int_{\T^3} \rho^m \nabla \cdot (\umm) \rmmm \,dx\cr
&\leq CM_0\lt( \|\rmmm\|_{L^2}^2 + \|\umm\|_{H^1}^2\rt).
\end{aligned}
\end{align*}
We apply the Gronwall's inequality to obtain
\[
\|\rmmm(t)\|_{L^2}^2 \ls \int_0^t \|\umm(s)\|_{H^1}^2 ds.
\]
We next find that
\[
\pa_t \ummm + \umm \cdot \nabla u^{m+1} + u^{m-1}\cdot \nabla \ummm = - \umm + \vmm.
\]
Thus we obtain
\begin{align}\label{diff-u-1}
\begin{aligned}
\frac12\frac{d}{dt}\|\ummm\|_{L^2}^2 &= -\int_{\T^3} (\umm \cdot \nabla u^{m+1})\cdot \ummm \,dx - \int_{\T^3} (u^{m-1}\cdot \nabla \ummm) \cdot \ummm \,dx\cr
&\quad -\int_{\T^3} \umm \cdot \ummm \,dx + \int_{\T^3} \vmm \cdot \ummm \,dx\cr
&\ls \|\umm\|_{L^2}^2 + \|\ummm\|_{L^2}^2 + \|\vmm\|_{L^2}^2.
\end{aligned}
\end{align}
Similarly, we can also deduce 
\bq\label{diff-u-2}
\frac{d}{dt}\|\nabla (\ummm)\|_{L^2}^2 \ls \|\umm\|_{H^1}^2 + \|\ummm\|_{H^1}^2 + \|\vmm\|_{H^1}^2.
\eq
We now combine \eqref{diff-u-1} and \eqref{diff-u-2} to get
\[
\|\ummm(t)\|_{H^1}^2 \ls \int_0^t \|\umm(s)\|_{H^1}^2 + \|\ummm(s)\|_{H^1}^2 + \|\vmm(s)\|_{H^1}^2 ds.
\]
Then we apply the Gronwall's inequality to find
\[
\|\ummm(t)\|_{H^1}^2 \ls \int_0^t \|\umm(s)\|_{H^1}^2 ds + \int_0^t \int_0^s \|\vmm(\tau)\|_{H^1}^2 d\tau ds,
\]
where we used
\[
\int_0^t \int_0^s \|\umm(\tau)\|_{H^1}^2 d\tau ds \leq T_0\int_0^t \|\umm(s)\|_{H^1}^2 ds.
\]
Hence we have
\bq\label{res-rhou}
\|\rmmm(t)\|_{L^2}^2 + \|\ummm(t)\|_{H^1}^2 \ls  \int_0^t \|\umm(s)\|_{H^1}^2 ds + \int_0^t \int_0^s \|\vmm(\tau)\|_{H^1}^2 d\tau ds.
\eq
We now estimate $\wmmm$ in $L^2$-norm. It follows from $\eqref{approx-sys}_3$ that $\wmmm$ satisfies
\begin{align}\label{diff-w}
\begin{aligned}
&A^0_m\pa_t \wmmm + \sum_{j=1}^3 A^j_m \pa_j \wmmm \cr
&\quad = -(A^0_m - A^0_{m-1})\pa_t w^m - \sum_{j=1}^3(A^j_m - A^j_{m-1})\pa_j w^m + A^0_mE_1^m - A^0_{m-1}E_1^{m-1}\cr
&\qquad + A^0_m(E_2^m - E_2^{m-1}) + (A^0_m - A^0_{m-1})E_2^{m-1}.
\end{aligned}
\end{align}
Note that 
\[
c_0^{-1}\mathbb{I}_{4 \times 4} \leq A^0_m \leq c_0 \mathbb{I}_{4 \times 4} \quad \mbox{for all} \quad m \geq 0,
\]
where $c_0$ is a positive constant independent of $m$. Then we use \eqref{diff-w} to get
\begin{align*}
\begin{aligned}
&\frac12\lt(\pa_t \lag A^0_m \wmmm,\wmmm \rag - \lag \pa_t A^0_m \wmmm,\wmmm\rag\rt)\cr
&\,\,\, + \frac12\sum_{j=1}^3 \lt( \pa_j \lag A^j_m \wmmm,\wmmm \rag - \lag \pa_jA^j_m \wmmm,\wmmm\rag\rt)\cr
&\quad = -\lag (A^0_m - A^0_{m-1})\pa_t w^m, \wmmm \rag - \sum_{j=1}^3 \lag (A^j_m - A^j_{m-1})\pa_j w^m, \wmmm\rag \cr
&\quad \quad +\lag A^0_mE_1^m - A^0_{m-1}E_1^{m-1}, \wmmm\rag + \lag A^0_m(E_2^m - E_2^{m-1}), \wmmm\rag + \lag (A^0_m - A^0_{m-1})E_2^{m-1}, \wmmm \rag.
\end{aligned}
\end{align*}
We notice that
\[
|A^0_m - A^0_{m-1}| \ls |\wmm|, \quad |A^j_m - A^j_{m-1}| \ls |\wmm|,
\]
\[
|E_2^m - E_2^{m-1}| \ls |\rmmm| + |\ummm| + |\vmmm| + |\nmmm|,
\]
and 
\begin{align*}
\begin{aligned}
\lag A^0_mE_1^m - A^0_{m-1}E_1^{m-1}, \wmmm\rag &= -\int_{\T^3} L\vmmm \cdot \vmmm dx \cr
&= -\mu\int_{\T^3} |\nabla \vmmm|^2 dx -(\mu + \lambda)\int_{\T^3} |\nabla \cdot \vmmm|^2 dx.
\end{aligned}
\end{align*}
Thus we obtain
\begin{align}\label{res-w}
\begin{aligned}
&\|\wmmm(t)\|_{L^2}^2  \leq c_0 \int_{\T^3} \lag A^0_m \wmmm, \wmmm\rag\big|_{t=t} \,dx\cr
&\quad \leq c_0\int_{\T^3} \lag A^0_m \wmmm,\wmmm\rag \big|_{t=0} \,dx\cr
&\qquad + C\int_0^t \lt(\|\wmm(s)\|_{L^2} + \|\rmm(s)\|_{L^2} + \|\umm(s)\|_{L^2}\rt)\|\wmmm(s)\|_{L^2}\,ds\cr
&\qquad - c_0\mu\int_0^t \|\nabla \vmmm(s)\|_{L^2}^2 ds - c_0(\mu + \lambda)\int_0^t \|\nabla \cdot \wmmm(s)\|_{L^2}^2 ds\cr
&\quad \leq C\int_0^t \|\wmmm(s)\|_{L^2}^2 ds + C\int_0^t \|\wmm(s)\|_{L^2}^2 ds + C\int_0^t \|\rmm(s)\|_{L^2}^2 ds\cr
&\qquad  + C\int_0^t \|\umm(s)\|_{L^2}^2 ds - c_0\mu\int_0^t \|\nabla \vmmm(s)\|_{L^2}^2 ds.
\end{aligned}
\end{align}
Now we collect the estimates \eqref{res-rhou} and \eqref{res-w} to find
\begin{align*}
\begin{aligned}
&\|\rmmm(t)\|_{L^2}^2 + \|\ummm(t)\|_{H^1}^2 + \|\wmmm(t)\|_{L^2}^2 + \int_0^t\|\nabla \vmmm(s)\|_{L^2}^2 ds \cr
&\qquad \ls \int_0^t \|\rmm(s)\|_{L^2}^2 ds + \int_0^t \|\umm(s)\|_{L^2}^2 ds + \int_0^t \|\wmm(s)\|_{L^2}^2 ds\cr
&\qquad \quad + \int_0^t \|\wmmm(s)\|_{L^2}^2 ds + \int_0^t \int_0^s \|\nabla \vmm(\tau)\|_{L^2}^2 d\tau ds,
\end{aligned}
\end{align*}
and this yields again
\[
\|\rmmm(t)\|_{L^2}^2 + \|\ummm(t)\|_{H^1}^2 + \|\wmmm(t)\|_{L^2}^2 + \int_0^t\|\nabla \vmmm(s)\|_{L^2}^2 ds \ls \frac{T_0^{m+1}}{(m+1)!}, 
\]
for $t \leq T_0.$ This completes the proof.
\end{proof}
\section{High-order estimates of $u,v$ and $n$}\label{app:high}
In this part, we provide the details of the proofs for Lemmas \ref{lem:u} - \ref{lem:n}.  
\subsection{The proof of Lemma \ref{lem:u}}
For $2 \leq k \leq s+2$, we get
\[
\nabla^k u_t = -\nabla^k(u\cdot\nabla u) - \nabla^k(u-v) = - u\cdot\nabla^{k+1}u - [\nabla^k,u\cdot \nabla]u - \nabla^k(u-v).
\]
Then we obtain
\begin{align*}
\begin{aligned}
\frac12\frac{d}{dt}\|\nabla^k u\|_{L^2}^2 &\leq \frac12\int_{\T^3} (\nabla \cdot u)|\nabla^k u|^2 dx + \|[\nabla^k, u\cdot \nabla]u\|_{L^2}\|\nabla^k u\|_{L^2} - \|\nabla^k u\|_{L^2}^2+ \|\nabla^k v\|_{L^2}\|\nabla^k u\|_{L^2}\cr
&\leq C\epsilon_1\|\nabla^k u\|_{L^2}^2 - \frac12\|\nabla^k u\|_{L^2}^2 + \frac12\|\nabla^k v\|_{L^2}^2\cr
&\leq -\lt(\frac12 - C\epsilon_1\rt)\|\nabla^k u\|_{L^2}^2 + \frac12\|\nabla^k v\|_{L^2}^2.
\end{aligned}
\end{align*}
This completes the proof.
\subsection{The proof of Lemma \ref{lem:v}}
Recall that $v$ satisfies 
\[
v_t + v \cdot \nabla v + \frac{\nabla p(n+1)}{n+1} + \frac{Lv}{n+1} = \frac{\rho}{n+1}(u-v).
\]
Then, for $2 \leq k \leq s+1$, it follows that
\[
\nabla^{k} v_t + \nabla^{k}(v\cdot \nabla v) + \nabla^{k}\lt(\frac{\nabla p(n+1)}{n+1}\rt) + \nabla^{k}\lt(\frac{Lv}{n+1}\rt) = \nabla^{k}\lt(\frac{\rho}{n+1}(u-v)\rt),
\]
and
\begin{align*}
\begin{aligned}
\nabla^{k} v_t &= - v \cdot \nabla^{k+1} v - [\nabla^{k},v\cdot \nabla]v + \nabla^{k}\lt(\frac{\rho}{n+1}(u-v)\rt)  - \nabla^{k} \lt(\frac{\nabla p(n+1)}{n+1}\rt) - \nabla^{k} \lt(\frac{Lv}{n+1}\rt).
\end{aligned}
\end{align*}
This yields that
\begin{align*}
\begin{aligned}
\frac12\frac{d}{dt}\|\nabla^{k} v\|_{L^2}^2 
&= -\int_{\T^3} (v\cdot \nabla^{k+1}v) \cdot \nabla^{k} v +  [\nabla^{k},v\cdot \nabla]v \cdot \nabla^{k} v - \nabla^{k}\lt(\frac{\rho}{n+1}(u-v)\rt) \cdot \nabla^{k} v \,dx\cr
&-\int_{\T^3} \nabla^{k} \lt(\frac{\nabla p(n+1)}{n+1}\rt) \cdot \nabla^{k} v + \nabla^{k} \lt(\frac{Lv}{n+1}\rt) \cdot \nabla^{k} v\,dx\cr
&=: \sum_{i=1}^5 I_i.
\end{aligned}
\end{align*}
Here $I_i,i=1,2,4$ are easily estimated as follows.
\begin{align*}
\begin{aligned}
I_1 &\leq \frac12\int_{\T^3} (\nabla \cdot v)|\nabla^{k} v|^2 dx \ls \epsilon_1 \|\nabla^{k} v\|_{L^2}^2,\cr
I_2 &\leq \|[\nabla^{k}, v \cdot \nabla]v\|_{L^2}\|\nabla^{k} v\|_{L^2} \ls \epsilon_1\|\nabla^{k} v\|_{L^2}^2,\cr
I_4 &= \int_{\T^3} \nabla^{k-1} \lt( \frac{\nabla p(n+1)}{n+1}\rt)\cdot \nabla^{k+1} v \,dx \cr
&=\int_{\T^3} \nabla^{k-1} \lt( \gamma(n+1)^{\gamma-2}\nabla n\rt)\cdot \nabla^{k+1} v \,dx \cr
&\leq C\sum_{1 \leq l \leq k-1}\int_{\T^3} |\nabla^l ((n+1)^{\gamma-2})||\nabla^{k-l} n||\nabla^{k+1}v|\,dx + C\int_{\T^3} |\nabla^k n| |\nabla^{k+1} v|\,dx\cr
&\leq C\sum_{1 \leq l \leq k-1}\|\nabla^l ((n+1)^{\gamma-2})\|_{H^1}\|\nabla^{k-l} n\|_{H^1}\|\nabla^{k+1}v\|_{L^2} + C\|\nabla^k n\|_{L^2}\|\nabla^{k+1}v\|_{L^2}\cr
&\leq C\|\nabla n\|_{H^{k-1}}\|\nabla^{k+1}v\|_{L^2}\cr
&\leq C\delta^{-1}\|\nabla n\|_{H^{k-1}}^2 + C\delta\|\nabla^{k+1}v\|_{L^2}^2,
\end{aligned}
\end{align*}
where $\delta > 0$ will be determined later. For the estimate of $I_3$, we obtain
\[
I_3 = \int_{\T^3} \nabla^{k-1} \lt( \frac{\rho}{n+1} (u-v)\rt) \cdot \nabla^{k+1} v \,dx \leq \lt\|\nabla^{k-1} \lt( \frac{\rho}{n+1} (u-v)\rt)\rt\|_{L^2}\|\nabla^{k+1}v\|_{L^2}.
\]
Note that
\begin{align*}
\begin{aligned}
\lt\|\nabla^{k-1} \lt( \frac{\rho}{n+1} (u-v)\rt)\rt\|_{L^2} &\leq \lt\|\frac{\rho}{n+1}\rt\|_{H^{k-1}}\|u-v\|_{L^\infty} + \lt\| \frac{\rho}{n+1}\rt\|_{L^\infty}\|u-v\|_{H^{k-1}}\cr
&\leq \lt(\|\rho\|_{H^{k-1}}\lt\|\frac{1}{n+1}\rt\|_{L^\infty} + \|\rho\|_{L^\infty}\lt\|\frac{1}{n+1}\rt\|_{H^{k-1}}\rt)\|u-v\|_{H^2} + C\epsilon_1\|u - v\|_{H^{k-1}}\cr
&\leq C\epsilon_1\lt(\|u-v\|_{H^2} + \|u - v\|_{H^{k-1}}\rt),
\end{aligned}
\end{align*}
Thus we find
\[
I_3 \leq C\epsilon_1\|u-v\|_{H^{k-1}}^2 + C\epsilon_1\|u-v\|_{H^2}^2 + C\epsilon_1\|\nabla^{k+1} v\|_{L^2}^2.
\]
Finally for the estimate $I_5$, we divide it into two parts:
\begin{align*}
\begin{aligned}
I_5 &= \int_{\T^3} \nabla^{k} v \cdot \nabla^{k}\lt( \frac{\mu}{n+1}\nabla \cdot \nabla v + \frac{\mu + \lambda}{n+1}\nabla \nabla \cdot v \rt)dx\cr
&=:I_5^1 + I_5^2.
\end{aligned}
\end{align*}
Then $I_5^1$ is estimated as follows.
\begin{align*}
\begin{aligned}
I_5^1 &= \int_{\T^3} \frac{\mu}{n+1}\nabla\cdot\nabla^{k+1}v \cdot \nabla^{k}v \,dx + 
\int_{\T^3}\nabla^{k}v \cdot \nabla^{k}\lt( \frac{\mu}{n+1}\rt) \nabla \cdot \nabla v \,dx\cr
&\quad + \mu\sum_{1\leq l \leq k-1} \binom{k}{l}\int_{\T^3}\nabla^{k}v \cdot \nabla^l\lt( \frac{1}{n+1}\rt)\nabla^{k-l}\nabla \cdot \nabla v \,dx\cr
&=:I_{5}^{1,1} + I_{5}^{1,2} + I_{5}^{1,3},\cr
\end{aligned}
\end{align*}
where
\begin{align*}
\begin{aligned}
I_{5}^{1,1}&= - \int_{\T^3} \nabla\lt( \frac{\mu}{n+1}\nabla^{k}v\rt)\cdot \nabla^{k+1} v \,dx\cr
&= \int_{\T^3} \lt( \frac{\mu\nabla n}{(n+1)^2} \cdot \nabla^{k}v\rt)\cdot \nabla^{k+1} v \,dx - \int_{\T^3} \frac{\mu}{n+1}|\nabla^{k+1} v|^2 dx\cr
&\leq C\epsilon_1\|\nabla^{k}v\|_{L^2}\|\nabla^{k+1}v\|_{L^2} - C{\mu}\|\nabla^{k+1} v\|_{L^2}^2\cr
&\leq C\epsilon_1\|\nabla^{k}v\|_{L^2}^2 - C(\mu - C\epsilon_1)\|\nabla^{k+1}v\|_{L^2}^2,\cr
I_{5}^{1,2}&\leq C\|\nabla^{k}v\|_{L^4}\|\nabla^2 v\|_{L^4}\lt\|\nabla^{k}\lt( \frac{1}{n+1}\rt) \rt\|_{L^2}\cr
&\leq C\|\nabla^{k}v\|_{H^1}\|\nabla^2 v\|_{H^1}\|n\|_{H^{k}}\cr
&\leq C\epsilon_1\|\nabla^{k+1}v\|_{L^2}^2 + C\epsilon_1\|\nabla^2 v\|_{H^1}^2,\cr
I_{5}^{1,3} &\ls \sum_{1\leq l \leq k-1}\|\nabla^{k}v\|_{L^4}\lt\| \nabla^l\lt( \frac{1}{n+1}\rt)\rt\|_{L^4}\|\nabla^{k+2-l} v\|_{L^2}\cr
&\ls \|\nabla^{k}v\|_{H^1}\|n\|_{H^{k}}\|\nabla^3 v\|_{H^{k-2}} \cr
&\leq C\epsilon_1\|\nabla^{k+1}v\|_{L^2}^2 + C\epsilon_1\|\nabla v\|_{H^{k-1}}^2.
\end{aligned}
\end{align*}
This yields 
\[
I_{5}^1 \leq C\epsilon_1\|\nabla v\|_{H^{k-1}}^2 - C(\mu - C\epsilon_1)\|\nabla^{k+1}v\|_{L^2}^2.
\]
Similarly, we can also obtain
\[
I_5^2 \leq C\epsilon_1\|\nabla v\|_{H^{k-1}}^2.
\]
This deduces
\begin{align*}
\begin{aligned}
\frac{d}{dt}\|\nabla^k v\|_{L^2}^2 + C(\mu - C\epsilon_1)\|\nabla^{k+1} v\|_{L^2}^2 \leq \frac{C}{\mu}\|\nabla n\|_{H^{k-1}}^2 + C\epsilon_1(\|\nabla^2 u\|_{L^2}^2 + \|\nabla^{k-1} u\|_{L^2}^2) + C\mathcal{S}_0(1),
\end{aligned}
\end{align*}
where we chose $\delta = \mu/2$ and used $\|\nabla v\|_{L^2} \leq C\|\nabla^l v\|_{L^2}$ for $l \in [1,k+1]$.
\subsection{The proof of Lemma \ref{lem:n}} Similar to the previous arguments in Lemma \ref{est-uvnh1}, we estimate
\begin{align*}
\begin{aligned}
&\frac{d}{dt}\int_{\T^3} \nabla^k n \cdot \lt( \frac{\nabla^k n}{2} + \frac{(n+1)^2}{2\mu + \lambda} \nabla^{k-1} v\rt) dx\cr
&\quad = \int_{\T^3} \nabla^k n_t \cdot \nabla^k n \,dx + \int_{\T^3} \frac{(1+n)^2}{2\mu + \lambda}\nabla^k n_t \cdot \nabla^{k-1}v\,dx\cr
&\qquad  + \int_{\T^3} \frac{(1+n)^2}{2\mu + \lambda}\nabla^k n \cdot \nabla^{k-1} v_t \,dx + \int_{\T^3} \nabla^{k-1}n \cdot \lt( \frac{2(1+n)}{2\mu + \lambda}n_t\nabla^{k-1}v \rt)dx\cr
&\quad =: \sum_{i=1}^4 J_i,
\end{aligned}
\end{align*}
for $2\leq k \leq s+1$. \newline
$\diamond$ Estimate of $J_1$: A straightforward computation yields that
\begin{align*}
\begin{aligned}
J_1 &= -\int_{\T^3} \nabla^k n \cdot \nabla^k\lt( \nabla n \cdot v + (1+n)\nabla \cdot v\rt)dx\cr
&= - \int_{\T^3} \nabla^k n \cdot \lt(v \cdot \nabla^{k+1}n\rt)dx - \int_{\T^3} [\nabla^k,v\cdot \nabla ]n \cdot \nabla^k n \,dx - \int_{\T^3} \nabla^k n \cdot \nabla^k \lt( (1+n)\nabla \cdot v\rt)dx\cr
&=\sum_{i=1}^3 J_1^i,
\end{aligned}
\end{align*}
where $J_1^i,i=1,2,3$ are estimated as follows.
\begin{align*}
\begin{aligned}
J_1^1 &\leq \|\nabla \cdot v\|_{L^\infty}\|\nabla^k n\|_{L^2}^2\leq C\epsilon_1\|\nabla^k n\|_{L^2}^2,\cr
J_1^2 &\leq \|[\nabla^k,v \cdot \nabla]n\|_{L^2}\|\nabla^k n\|_{L^2} \leq C\epsilon_1\|\nabla^k n\|_{L^2}^2,\cr
J_1^3 &= -\int_{\T^3}\nabla^{k-1}\pa_i n \lt( \pa_i n \nabla^{k-1}\pa_j v_j + (n+1)\nabla^{k-1}\pa_{ij}v_j\rt)dx- \int_{\T^3} \nabla^{k-1}\pa_i n \lt( (\nabla^{k-1}\pa_i n)\pa_j v_j + \nabla^{k-1}n \pa_{ij}v_j\rt)dx\cr
&\quad - \delta_{k-2}\sum_{1 \leq l < k-1}\binom{k-1}{l}\int_{\T^3} \nabla^{k-1}\pa_i n \lt( \nabla^l \pa_i n \nabla^{k-1-l}\pa_j v_j + \nabla^l n \nabla^{k-1-l}\pa_{ij}v_j\rt)dx\cr
&\leq C\|\nabla^k n\|_{L^2}\|\nabla n\|_{L^\infty}\|\nabla^k v\|_{L^2} + \int_{\T^3}(n+1)\nabla^{k-1}\pa_{ij} n \cdot \nabla^{k-1}\pa_i v_j\,dx + C\|\nabla^k n\|_{L^2}\|\nabla^{k-1}n\|_{L^2}\|\nabla^2 v\|_{L^\infty}\cr
&\quad + C\|\nabla^k n\|_{L^2}^2\|\nabla v\|_{L^\infty}+ C\delta_{k-2}\lt(\|\nabla^k n\|_{L^2}\|\nabla^2 n\|_{H^{k-2}}\|\nabla^2 v\|_{H^{k-2}} + \|\nabla^k n\|_{L^2}\|\nabla n\|_{H^{k-2}}\|\nabla^3 v\|_{H^{k-2}}\rt)\cr
&\leq C\epsilon_1\|\nabla^k n\|_{L^2}^2 + C\epsilon_1\|\nabla^k v\|_{L^2}^2 + C\epsilon_1\delta_{k-2}\lt( \|\nabla^2 n\|_{H^{k-2}}^2 + \|\nabla^k n\|_{L^2}^2 + \|\nabla^3 v\|_{H^{k-2}}^2\rt)\cr
&\quad + \int_{\T^3}(n+1)\nabla^{k-1}\pa_{ij} n \cdot \nabla^{k-1}\pa_i v_j\,dx.
\end{aligned}
\end{align*}
Here we used the following estimates for the first term in $J_1^3$.
\begin{align*}
\begin{aligned}
&-\int_{\T^3}\nabla^{k-1}\pa_i n \lt( \pa_i n \nabla^{k-1}\pa_j v_j + (n+1)\nabla^{k-1}\pa_{ij}v_j\rt)dx\cr
&\quad = -\int_{\T^3} (\nabla^{k-1}\pa_i n)\pa_i n \nabla^{k-1}\pa_j v_j \,dx + \int_{\T^3} \lt( (\nabla^{k-1}\pa_{ij}n)(n+1) + (\nabla^{k-1}\pa_i n)\pa_j n \rt)\pa_i v_j\,dx\cr
&\quad \leq C\|\nabla^k n\|_{L^2}\|\nabla n\|_{L^\infty}\|\nabla^k v\|_{L^2} + \int_{\T^3}(n+1)\nabla^{k-1}\pa_{ij} n \cdot \nabla^{k-1}\pa_i v_j\,dx.
\end{aligned}
\end{align*}
Thus we obtain
\[
J_1 \leq C\epsilon_1\|\nabla^2 n\|_{H^{k-2}}^2 + C\epsilon_1\|\nabla^2 v\|_{H^{k-1}}^2+ \int_{\T^3} (1+n)\nabla^{k-1}\pa_{ij}n \cdot \nabla^{k-1}\pa_i v_j \,dx.
\]
$\diamond$ Estimate of $J_2$: For this, we get
\begin{align*}
\begin{aligned}
J_2 &= -\frac{1}{2\mu + \lambda}\int_{\T^3} (1+n)^2 \nabla^{k-1}v \cdot \lt( \nabla^k(\nabla n \cdot v) + \nabla^k((1+n)\nabla\cdot v)\rt) dx\cr
&=\frac{1}{2\mu + \lambda}\int_{\T^3} \lt( 2(1+n)\nabla n \cdot \nabla^{k-1}v + (1+n)^2\nabla^k v\rt)\cdot \nabla^{k-1}(\nabla n \cdot v)\,dx\cr
&=\frac{1}{2\mu + \lambda}\int_{\T^3} \lt( 2(1+n)\nabla n \cdot \nabla^{k-1}v + (1+n)^2\nabla^k v\rt)\cdot \nabla^{k-1}((1+n)\nabla\cdot v)\, dx\cr
&=:J_2^1 + J_2^2.\cr
\end{aligned}
\end{align*}
Here $J_2^1$ is estimated as follows.
\begin{align*}
\begin{aligned}
J_2^1 &\leq C\sum_{1 \leq l \leq k-1}\int_{\T^3} \lt( |\nabla n||\nabla^{k-1} v| + |\nabla^k v|\rt)|\nabla^l v| |\nabla^{k-l} n|\,dx+ C\int_{\T^3} \lt( |\nabla n||\nabla^{k-1}v| + |\nabla^k v| \rt)|v||\nabla^k n|\,dx\cr
&\leq C\sum_{1 \leq l \leq k-1}\lt( \epsilon_1 \|\nabla^{k-1} v\|_{L^2} + \|\nabla^k v\|_{L^2}\rt)\|\nabla^l v\|_{H^1}\|\nabla^{k-l} n\|_{H^1} + C\epsilon_1\lt(\|\nabla^{k-1} v\|_{L^2} + \|\nabla^k v\|_{L^2}\rt)\|\nabla^k n\|_{L^2}\cr
&\leq C\epsilon_1\|\nabla^k v\|_{L^2}^2 + C\epsilon_1\|\nabla n\|_{H^{k-1}}^2,\cr
\end{aligned}
\end{align*}
where we again used $\|\nabla v\|_{L^2} \leq C\|\nabla^l v\|_{L^2}$ for $l \in [1,k+1]$ due to $k \geq 2$. Similarly, we also find
\begin{align*}
\begin{aligned}
J_2^2 &\leq C\sum_{1 \leq l \leq k-1}\int_{\T^3} \lt(|\nabla n||\nabla^{k-1} v| + |\nabla^k v|\rt)|\nabla^l n| |\nabla^{k-l}v|\,dx+ C\int_{\T^3} \lt(|\nabla n||\nabla^{k-1} v| + |\nabla^k v|\rt)|1+n| |\nabla^k v|\,dx\cr
&\leq C\sum_{1 \leq l \leq k-1}\lt(\epsilon_1\|\nabla^{k-1}v\|_{L^2} + \|\nabla^k v\|_{L^2}\rt)\|\nabla^l n\|_{H^1}\|\nabla^{k-l}v\|_{H^1} + C\epsilon_1\|\nabla^{k-1} v\|_{L^2}\|\nabla^k v\|_{L^2} + C\|\nabla^k v\|_{L^2}^2\cr
&\leq C\epsilon_1\lt(\|\nabla^{k-1} v\|_{L^2} + \|\nabla^k v\|_{L^2}\rt)\|\nabla v\|_{H^{k-1}} + C\|\nabla^k v\|_{L^2}^2\cr
&\leq C\|\nabla^k v\|_{L^2}^2,
\end{aligned}
\end{align*}
 Thus we have
\[
J_2 \leq C\|\nabla^k v\|_{L^2}^2 + C\epsilon_1\|\nabla n\|_{H^{k-1}}^2.
\]

\noindent $\diamond$ Estimate of $J_3$: Note that
\begin{align*}
\begin{aligned}
J_3 &= -\frac{1}{2\mu + \lambda}\int_{\T^3} (1+n)^2\nabla^k n\cdot\lt(v \cdot \nabla^k v - [\nabla^{k-1}, v \cdot \nabla]v - \nabla^{k-1}\lt( \frac{\rho}{1+n}(u-v)\rt) \rt)dx\cr
&\quad -\frac{1}{2\mu + \lambda}\int_{\T^3}(1+n)^2\nabla^k n\cdot\lt( \nabla^{k-1}\lt( \frac{\nabla p(n+1)}{n+1}\rt) + \nabla^{k-1}\lt( \frac{Lv}{n+1}\rt) \rt)dx\cr
&=:\sum_{i=1}^5 J_3^i,
\end{aligned}
\end{align*}
where $J_3^i,i=1,2,3$ are easily estimated as follows.
\begin{align*}
\begin{aligned}
J_3^1 &\ls \|v\|_{L^\infty}\|\nabla^k n\|_{L^2}\|\nabla^k v\|_{L^2} \leq C\epsilon_1\|\nabla^k n\|_{L^2}\|\nabla^k v\|_{L^2},\cr
J_3^2 &\ls \|\nabla v\|_{L^\infty}\|\nabla^k n\|_{L^2}\|\nabla^k v\|_{L^2} \leq C\epsilon_1\|\nabla^k n\|_{L^2}\|\nabla^{k} v\|_{L^2},\cr
J_3^3 &\ls \|\nabla^k n\|_{L^2}\lt\|\nabla^{k-1}\lt(\frac{\rho}{1+n}(u-v)\rt)\rt\|_{L^2} \leq C\epsilon_1\|\nabla^k n\|_{L^2}\|u-v\|_{H^{k}},\cr
\end{aligned}
\end{align*}
due to $\|\nabla v\|_{L^2} \leq C\|\nabla^k v\|_{L^2}$ for $k \in [1,s+1]$. 
For the estimate of $J_3^4$, we get
\begin{align*}
\begin{aligned}
J_3^4 &= - \frac{\gamma}{2\mu + \lambda}\int_{\T^3} (1+n)^2\nabla^k n \cdot\nabla^{k-1}\lt( (1+n)^{\gamma-2}\nabla n \rt)dx\cr
&=-\frac{\gamma}{2\mu + \lambda}\sum_{1\leq l \leq k-1}\binom{k-1}{l}\int_{\T^3}(1+n)^2 \nabla^k n \cdot \lt(\nabla^l\lt( (1+n)^{\gamma-2}\rt)\nabla^{k-l}n\rt)dx\cr
&\quad - \frac{\gamma}{2\mu + \lambda}\int_{\T^3}(1+n)^\gamma |\nabla^k n|^2 dx\cr
&=: J_3^{4,1} + J_3^{4,2}.
\end{aligned}
\end{align*}
Here $J_3^{4,1}$ is estimated as follows.
\begin{align*}
\begin{aligned}
J_3^{4,1} &\ls \sum_{1\leq l \leq k-1}\|\nabla^k n\|_{L^2}\|\nabla^l (1+n)^{\gamma-2}\|_{L^4}\|\nabla^{k-l}n\|_{L^4}\cr
&\ls \|\nabla^k n\|_{L^2}\|\nabla (1+n)^{\gamma-2}\|_{H^{k-1}}\|\nabla n\|_{H^{k-1}}\cr
&\leq C\epsilon_1\|\nabla^k n\|_{L^2}\|\nabla n\|_{H^{k-1}}\cr
&\leq C\epsilon_1\|\nabla n\|_{H^{k-1}}^2,
\end{aligned}
\end{align*}
where we used
\begin{align*}
\begin{aligned}
\|\nabla (1+n)^{\gamma-2}\|_{H^{k-1}} &= \lt\| \nabla\lt( \frac{p}{(1+n)^2}\rt)\rt\|_{H^{k-1}}\cr
&\ls \lt\| \frac{\nabla p}{(1+n)^2}\rt\|_{H^{k-1}} + \lt\|p\frac{\nabla n}{(1+n)^3} \rt\|_{H^{k-1}}\cr
&\ls \|\nabla p\|_{H^{k-1}}\lt\|\frac{1}{(1+n)^2} \rt\|_{L^\infty} + \|\nabla p\|_{L^\infty}\lt\|\frac{1}{(1+n)^2} \rt\|_{H^{k-1}}\cr
&\quad + \|p\|_{H^{k-1}}\lt\|\frac{\nabla n}{(1+n)^3} \rt\|_{L^\infty} + \|p\|_{L^\infty}\lt\| \frac{\nabla n}{(1+n)^3}\rt\|_{H^{k-1}}\cr
&\leq C\epsilon_1.
\end{aligned}
\end{align*}
For the $J_3^{4,2}$, we have 
\[
J_3^{4,2} \leq - \frac{C\gamma}{2\mu + \lambda}\|\nabla^k n\|_{L^2}^2,
\]
due to $\|n\|_{L^\infty }\leq \epsilon_1 \ll 1$. Thus we get
\[
J_3^4 \leq C\epsilon_1\|\nabla n\|_{H^{k-1}}^2 - \frac{C\gamma}{2\mu + \lambda}\|\nabla^k n\|_{L^2}^2.
\] 
We also find
\begin{align*}
\begin{aligned}
J_3^5 &= \frac{1}{2\mu + \lambda}\int_{\T^3} (1+n)^2\nabla^k n \cdot\lt( \frac{\mu\nabla^{k-1}(\nabla \cdot \nabla v)}{n+1} + \frac{\mu + \lambda}{n+1}\nabla^{k-1}\nabla \nabla \cdot v\rt) dx\cr
& \quad + \frac{1}{2\mu+\lambda}\sum_{1\leq l \leq k-1}\binom{k-1}{l}\int_{\T^3} (1+n)^2 \nabla^k n \nabla^l\lt( \frac{\mu}{n+1}\rt) \nabla^{k-1-l}(\nabla \cdot \nabla v)\,dx\cr
&\quad + \frac{1}{2\mu + \lambda}\sum_{1\leq l \leq k-1}\binom{k-1}{l}\int_{\T^3} (1+n)^2 \nabla^k n \nabla^l\lt( \frac{\mu + \lambda}{n+1}\rt) \nabla^{k-1-l}(\nabla \nabla \cdot v)\,dx\cr
&=:\sum_{i=1}^3 J_3^{5,i},
\end{aligned}
\end{align*}
where $J_3^{5,i},i=1,2,3$ are estimated as follows.
\begin{align*}
\begin{aligned}
J_3^{5,1} &= \frac{\mu}{2\mu + \lambda}\int_{\T^3}(1+n)\nabla^k n \cdot \nabla^{k-1}(\nabla\cdot \nabla v)dx+ \frac{\mu+\lambda}{2\mu + \lambda}\int_{\T^3} (1+n)\nabla^k n \cdot \nabla^{k-1}(\nabla \nabla \cdot v) dx\cr
&=-\frac{\mu}{2\mu + \lambda}\int_{\T^3} \lt( \pa_j n \nabla^{k-1} \pa_i n + (1+n)\nabla^{k-1}\pa_{ij}n \rt) \nabla^{k-1}\pa_j v_i\,dx\cr
&\quad - \frac{\mu + \lambda}{2\mu + \lambda}\int_{\T^3} \lt( \pa_j n \nabla^{k-1}\pa_i n + (1+n)\nabla^{k-1}\pa_{ij}n \rt)\nabla^{k-1}\pa_i v_j\,dx\cr
&= -\frac{1}{2\mu + \lambda}\int_{\T^3} \lt( \mu \pa_j n \nabla^{k-1}\pa_i n + (\mu + \lambda)\pa_i n \nabla^{k-1}\pa_j n\rt) \nabla^{k-1}\pa_j v_i\, dx\cr
&\quad - \int_{\T^3} (1+n)\nabla^{k-1}\pa_{ij} n \cdot \nabla^{k-1}\pa_j v_i\, dx\cr
&\leq C\|\nabla^k n\|_{L^2}\|\nabla^k v\|_{L^2}\|\nabla n\|_{L^\infty} - \int_{\T^3} (1+n)\nabla^{k-1}\pa_{ij} n \cdot \nabla^{k-1}\pa_j v_i\, dx\cr
&\leq C\epsilon_1\|\nabla^k n\|_{L^2}^2 + C\epsilon_1\|\nabla^k v\|_{L^2}^2 - \int_{\T^3} (1+n)\nabla^{k-1}\pa_{ij} n \cdot \nabla^{k-1}\pa_j v_i\, dx,
\end{aligned}
\end{align*}
and
\begin{align*}
\begin{aligned}
J_3^{5,2} + J_3^{5,3} &\ls \sum_{1\leq l \leq k-1}\|\nabla^k n\|_{L^2}\lt\| \nabla^l\lt( \frac{1}{n+1}\rt)\rt\|_{L^4}\|\nabla^{k+1-l} v\|_{L^4}\cr
&\ls \|\nabla^k n\|_{L^2}\lt\|\nabla\lt( \frac{1}{n+1}\rt)\rt\|_{H^{k-1}}\|\nabla^2 v\|_{H^{k-1}}\cr
&\leq C\epsilon_1\|\nabla^k n\|_{L^2}\|\nabla^2 v\|_{H^{k-1}}\cr
&\leq C\epsilon_1\|\nabla^k n\|_{L^2}^2 + C\epsilon_1\|\nabla v\|_{H^k}^2.
\end{aligned}
\end{align*}
This yields 
\[
J_3^5 \leq C\epsilon_1\|\nabla^k n\|_{L^2}^2 + C\epsilon_1\|\nabla v\|_{H^k}^2- \int_{\T^3} (1+n)\nabla^{k-1}\pa_{ij} n \cdot \nabla^{k-1}\pa_j v_i\, dx,
\]
and thus we have
\begin{align*}
\begin{aligned}
J_3 &\leq C\epsilon_1\|\nabla n\|_{H^{k-1}}^2 + C\epsilon_1\|\nabla v\|_{H^k}^2 + C\epsilon_1\|u - v\|_{H^{k}}^2 - \frac{C\gamma}{2\mu+\lambda}\|\nabla^k n\|_{L^2}^2- \int_{\T^3} (1+n)\nabla^{k-1}\pa_{ij} n \cdot \nabla^{k-1}\pa_j v_i\, dx.
\end{aligned}
\end{align*}
$\diamond$ Estimate of $J_4$: Since
\[
J_4 = -\frac{2}{2\mu + \lambda}\int_{\T^3} (1+n)\nabla^{k-1} n \cdot \nabla^{k-1}v \lt( \nabla n \cdot v + (1+n)\nabla \cdot v\rt)dx,
\]
we easily obtain
\begin{align*}
\begin{aligned}
J_4 &\ls \|\nabla^{k-1}n\|_{L^2}\|\nabla^{k-1}v\|_{L^2}\lt( \|\nabla n \cdot v\|_{L^\infty} + \|\nabla v\|_{L^\infty}\rt) \leq C\epsilon_1\|\nabla^{k-1}n\|_{L^2}^2 + C\epsilon_1\|\nabla^{k-1}v\|_{L^2}^2
\end{aligned}
\end{align*}
We finally combine all estimates above to have
\begin{align*}
\begin{aligned}
&\frac{d}{dt}\int_{\T^3} \nabla^k n \cdot \lt( \frac{\nabla^k n}{2} + \frac{(n+1)^2}{2\mu + \lambda} \nabla^{k-1} v\rt) dx + \frac{C\gamma}{2\mu + \lambda}\|\nabla^k n\|_{L^2}^2\cr
&\qquad \leq C\epsilon_1\|\nabla n\|_{H^{k-1}}^2 + C\epsilon_1\|\nabla v\|_{H^k}^2 + C\epsilon_1\|u-v\|_{H^{k}}^2 + C\|\nabla^k v\|_{L^2}^2,
\end{aligned}
\end{align*}
and this deduces 
\begin{align*}
\begin{aligned}
&\frac{d}{dt}\int_{\T^3} \nabla^k n \cdot \lt( \frac{\nabla^k n}{2} + \frac{(n+1)^2}{2\mu + \lambda} \nabla^{k-1} v\rt) dx + C_3\|\nabla^k n\|_{L^2}^2\cr
&\qquad \leq C\epsilon_1\|\nabla v\|_{H^k}^2 + C\epsilon_1\|\nabla^2 u\|_{H^{k-2}}^2 + C\|\nabla^k v\|_{L^2}^2 + C\mathcal{S}_0(1),
\end{aligned}
\end{align*}
due to $\|\nabla n\|_{L^2} \leq C\|\nabla^l n\|_{L^2}$ for $l \in [1,k]$.

%
%
%
%

\section*{Acknowledgments}
Y.-P. Choi was partially supported by Basic Science Research Program through the National Research Foundation of Korea (NRF) funded by the Ministry of Education, Science and Technology (2012R1A6A3A030 -39496) and Engineering and Physical Sciences Research Council(EP/K00804/1). Y.-P. Choi also acknowledges the support of the ERC-Starting Grant HDSPCONTR ``High-Dimensional Sparse Optimal Control''. 
B. Kwon was supported by Basic Science Research Program through the National Research Foundation of Korea(NRF) funded by the Ministry of Science, ICT and Future Planning (2015R1C1A1A02037662)

%
%
%
%

\end{document}